\documentclass[a4paper,12pt]{article}

\usepackage[numbers,sort&compress]{natbib}

\usepackage{mathtools}
\usepackage{authblk}
\usepackage[title]{appendix}
\usepackage{geometry}
\usepackage{subfiles}
\usepackage{caption}
\usepackage{subcaption}
\usepackage{setspace}
\usepackage{amsthm}
\usepackage{amsmath}
\usepackage{amssymb}

\usepackage[svgnames]{xcolor} 

\usepackage{ulem}

\usepackage{mathrsfs}

\numberwithin{equation}{section}

\usepackage{float}
\usepackage{fancyhdr}
\usepackage{graphicx}
\usepackage[colorlinks,            linkcolor=black,            anchorcolor=black,            citecolor=black            ]{hyperref}
\usepackage{listings}
 \usepackage{tikz}
 \usetikzlibrary{arrows,decorations.pathmorphing,backgrounds,positioning,fit,petri,arrows.meta,bending,positioning}

\newtheorem{proposition}{Proposition}[section]
\newtheorem{theorem}{Theorem}[section]
\newtheorem{lemma}{Lemma}[section]
\newtheorem{definition}{Definition}[section]

\newtheorem{remark}{Remark}[section]

\newtheorem*{theorem*}{Theorem}
\numberwithin{equation}{section}

\DeclareMathOperator\init{init}
\DeclareMathOperator\myin{in}

\newcommand{\p}{\partial}
\newcommand{\beq}{\begin{equation}}
\newcommand{\eeq}{\end{equation}}

\newcommand{\pP}{\mathcal{P}}

\newcommand{\bueq}{\begin{equation*}}
\newcommand{\eueq}{\end{equation*}}

\newcommand{\bthm}{\begin{theorem}}
	\newcommand{\ethm}{\end{theorem}}

\newcommand{\eps}{\epsilon}

\providecommand{\keywords}[1]
{
	\small	
	\textbf{{Keywords:}} #1
}
\providecommand{\msc}[1]
{
	\small	
	\textbf{{Mathematics Subject Classification:}} #1
}

\title{A voltage-conductance kinetic system from neuroscience: probabilistic reformulation and exponential ergodicity }%

\author{Xu'an Dou\thanks{Beijing International Center for Mathematical Research, Peking University, Beijing, 100871, China (dxa@pku.edu.cn)}\quad\quad Fanhao Kong\thanks{School of Mathematical Sciences, Peking University, Beijing, 100871, China (fanhaokong@pku.edu.cn)}\quad\quad Weijun Xu\thanks{Beijing International Center for Mathematical Research, Peking University, Beijing, 100871, China (weijunxu@bicmr.pku.edu.cn)}\quad\quad   Zhennan Zhou\thanks{Beijing International Center for Mathematical Research, Peking University, Beijing, 100871, China (zhennan@bicmr.pku.edu.cn).}}

\date{\today}
\begin{document}
\maketitle
\begin{abstract}
The voltage-conductance kinetic equation for an ensemble of neurons has been studied by many scientists and mathematicians, while its rigorous analysis is still at a premature stage. In this work, we obtain for the first time the exponential convergence to the steady state of this kinetic model in the linear setting. Our proof is based on a probabilistic reformulation, which allows us to investigate microscopic trajectories and bypass the difficulties raised by the special velocity field and boundary conditions in the macroscopic equation. We construct an associated stochastic process, for which proving the minorization condition becomes tractable, and the exponential ergodicity is then proved using Harris' theorem. 
\end{abstract}


\msc{35B40, 35Q84, 35Q92, 37A25, 92B20}







\section{Introduction}

We consider a kinetic equation from neuroscience, which is referred to as the voltage-conductance equation. It describes an ensemble of neurons via  $p(t,v,g)$, the probability density to find a neuron with voltage $v\in(V_R,V_F)$ and conductance $g\in(0,+\infty)$ at time $t$. Its evolution is governed by
	\begin{equation}\label{eq:FP-PDE}
	    \p_tp+\p_v(J(v,g)p)=\p_g((g-g_{\myin})p)+a\p_{gg}p,\quad t>0,v\in(V_R,V_F),g>0.
	\end{equation}
This model was originally proposed in \cite{cai2004effective,cai2006kinetic} and it has many successful applications in neuroscience \cite{cai2004effective,Cai2004embeded,cai2006kinetic,rangan2006maximum,rangan2007numerical,kovavcivc2009fokker,Cai2012307} as a computational tool or a theoretical framework. Nevertheless, its mathematical study is still at a premature stage due to its specific structures, which we shall elaborate as follows.

In $g$-direction, the equation \eqref{eq:FP-PDE} has a usual Fokker-Planck operator for the Ornstein-Uhlenbeck (OU) process on the half line, with the no-flux or reflective boundary condition at $g=0$: 
\begin{equation}\label{bc-g-pde}
        (g-g_{\myin})p+a\p_gp=0,\qquad g=0\;, \; \; v\in(V_R,V_F)\;, \; \; t>0
\end{equation} for mass conservation. Here, the parameters $g_{\myin}>0$ and $a>0$ are taken to be constants for simplicity. 

In $v$-direction, the equation structure is more complicated. The velocity field $J$ is given by
	\begin{equation}\label{eq:velocity-v}
	    J(v,g)=g_L(V_R-v)+g(V_E-v),
	\end{equation}
	where the fixed parameters satisfy
	\begin{equation}
	    	    g_L>0\;, \qquad V_E>V_F>V_R\;.
	\end{equation}
In \eqref{eq:velocity-v}, the first term  $g_L(V_R-v)$ models the leaky effect, which drives the voltage to a lower value $V_R$. Here $g_L>0$ is the constant leaky conductance. The second term $g(V_E-v)$ drives the voltage to the firing potential $V_F$. Here $V_E > V_F$ is the excitatory reversal potential. In other words, the velocity field manifests a competition between two mechanisms, where the strength of the former is fixed to be $g_L$, and the strength of the latter is depicted by the conductance variable $g$.

    A neuron spikes when its voltage reaches $V_F$, which happens only if $J(V_F,g)>0$, or equivalently $g>g_F:=g_L({V_F-V_R})/({V_E-V_F})$. In this case, the voltage is immediately reset to $V_R$ after the spike. This leads to the boundary condition
    \begin{equation}\label{bc-v-2-pde}
        J(V_R,g) \, p(t,V_R,g) = J(V_F,g) \, p(t,V_F,g), \quad g>g_F\;,
    \end{equation}
    matching the fluxes at $v=V_R$ and $v=V_F$. On the other hand, for $0<g\leq g_F$, we have $J(V_R,g)>0$ and $J(V_F,g)\leq 0$. So we impose the Dirichlet boundary condition
\begin{equation}\label{bc-v-1-pde}
    0=p(t,V_R,g)=p(t,V_F,g),\quad g\leq g_F.
\end{equation} Note that the flux equality \eqref{bc-v-2-pde} actually holds for all $g>0$, but the boundary condition changes type as $J(V_F,g)$ changes sign.     See Figure~\ref{fig:J-bc} for an illustration.

    \begin{figure}[htbp]
        \centering        \includegraphics[width=0.5\linewidth]{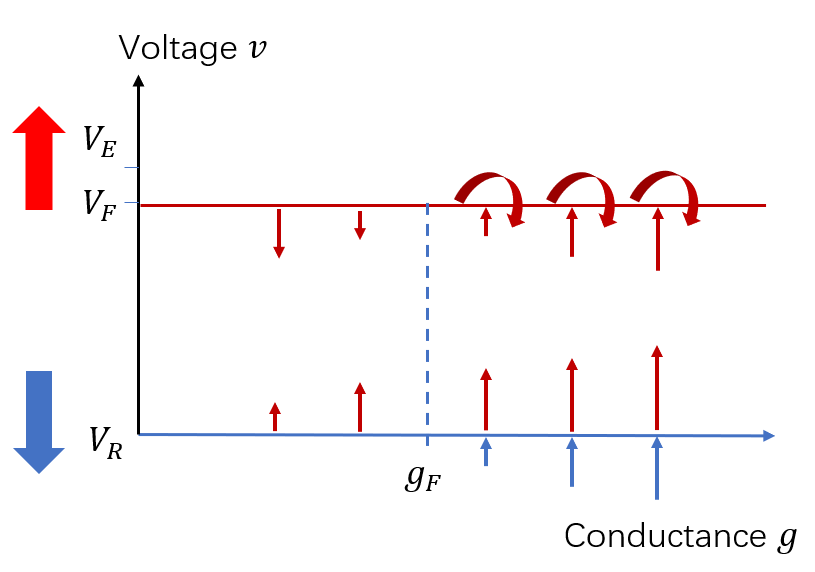}
        \caption{\footnotesize Illustration of the velocity field $J(v,g)$ and the boundary condition in $v$. { Red arrow} represents the velocity field $J$. {Blue arrow} indicates that the voltage $v$ is reset to $V_R$ after its arrival at $V_F$ in $\{g>g_F\}$. For $g\leq g_F$, the zero Dirichlet boundary condition is imposed.}
        \label{fig:J-bc}
    \end{figure}
We complement the system with an initial condition
\begin{equation}\label{ic-pde}
p(0,v,g)=p_{\init}(v,g),
\end{equation} which is assumed to be a probability density. The boundary conditions \eqref{bc-g-pde},\eqref{bc-v-2-pde} and \eqref{bc-v-1-pde} ensures that the total mass of $p$ is conserved by the dynamics. 

We refer to \cite{cai2006kinetic} for a comprehensive biological interpretation of the system. Note that taking $g_{\myin}>0$ and $a>0$ to be fixed constants simplifies the equation to the linear case. In more general situations, the two parameters can depend on time and even the solution $p$ in a nonlinear way. The nonlinearity can give rise to diverse phenomena, such as the periodic solution observed numerically in \cite{caceres2011numerical}. Nevertheless, even in the linear setting, the velocity field and the boundary condition in $v$ already brings difficulties in rigorous analysis. 

In a pioneering work (\cite{perthame2013voltage}) in the mathematical analysis of the voltage-conductance equation \eqref{eq:FP-PDE}, the authors study the steady state problem and derive several global bounds for the dynamical problem. Subsequent developments can be classified into two branches: either directly investigating the linear structure of \eqref{eq:FP-PDE} \cite{dpsz2022dcds,dou2022exponential}, or studying a nonlinear but simplified model \cite{perthame2018derivation,kim2021fast,CDZ2022}. In particular, several asymptotic limits and limiting models are studied in \cite{perthame2018derivation,kim2021fast}, and \cite{CDZ2022} fully characterizes the long time behavior of a variant model with motivations to understand the emergence of the periodic solution in the original model. In all these works above, the dynamics in $v$, in particular the boundary conditions \eqref{bc-v-2-pde}-\eqref{bc-v-1-pde} are simplified. To exemplify the difficulty of the original equation even in the linear case, we note that since the $L^{(8/7)^{-}}$ integrability of the steady state was obtained in \cite{perthame2013voltage}, the integrability index had not been improved until the recent $L^{\infty}$ estimate \cite{dpsz2022dcds}.

The primary goal of this paper is to explore the exponential convergence to the steady state for the linear equation \eqref{eq:FP-PDE}, which has been an open question since \cite{perthame2013voltage}. In \cite{perthame2013voltage}, only the convergence for the $g$-marginal is shown. Recently, qualitative convergence for the full solution $p(t,v,g)$ is established in \cite{dpsz2022dcds} via a compactness method without a convergence rate. A simplified model is proposed and studied in \cite{dou2022exponential} via the resolvent estimate. To the best of our knowledge, whether the solution of \eqref{eq:FP-PDE} converges to the steady state exponentially fast is not known before. 

To illustrate the difficulty, we note that the diffusion in \eqref{eq:FP-PDE} appears only in $g$ direction, while in $v$ direction there is only the convection term $\p_v(J(v,g)p)$. The situation resembles the classical kinetic Fokker-Planck equation, where the voltage $v$ is like the position variable, and the conductance $g$ is like the velocity variable. To prove convergence to the equilibrium, one needs to exploit the interaction between the convection in $v$ and the diffusion in $g$, which may ``pass'' the dissipation from the $g$ direction to the $v$ direction. This is of course not a new challenge, as it appears in the analysis of many classical kinetic models, known as hypocoercivity \cite{villani2009hypocoercivity}. Indeed, various hypocoercivity methods have been developed in the literature to address the convergence to the equilibrium for kinetic models ~\cite{villani2009hypocoercivity,SebastienGadat2013KRM,dolbeault2015hypocoercivity,baudoin2017bakry}. However, it seems difficult to apply many existing hypocoercivity methods to \eqref{eq:FP-PDE} due to its specific structures.  

The equation \eqref{eq:FP-PDE} is distinguished from many classical kinetic models by the velocity field $J(v,g)$ and the boundary conditions \eqref{bc-v-2-pde}-\eqref{bc-v-1-pde} in $v$. As defined in \eqref{eq:velocity-v}, the velocity field $J(v,g)$ can not be written in a separable form like $f(v)h(g)$, which results in several difficulties. Most importantly, it causes the steady state not to have a local equilibrium structure. To be more specific, the steady state $p_{\infty}(v,g)$ can not be written as a separable form $\theta(v)M(g)$, where $M(g)$ is in the kernel of the Fokker-Planck operator in $g$ (sometimes known as the local Maxwellian in classical kinetic models). Therefore, hypocoercivity methods based on the local equilibrium could not be applied to \eqref{eq:FP-PDE}. 

Closely related to $J(v,g)$, the boundary conditions in $v$ (\eqref{bc-v-2-pde} and \eqref{bc-v-1-pde}) also raise many difficulties. They make existing hypercoercivity methods inapplicable, which are designed only for problems in the whole space or on periodic domains. 

We remark that these difficulties not only appear at the technical level, but are also essential to the dynamics, e.g., to the emergence of periodic solutions in the nonlinear case. Indeed, once $J(v,g)$ is simplified to $g(V_E-v)$ (i.e. the leaky mechanism is removed and the boundary conditions are simplified accordingly), \cite{CDZ2022} excludes the possibility of having periodic solutions in the resulting nonlinear but simplified model. Hence, investigating the original linear equation \eqref{eq:FP-PDE} plays an important role in understanding the various intriguing dynamics in the nonlinear case.

In this work, we aim to prove the exponential convergence to the steady state of the voltage-conductance equation \eqref{eq:FP-PDE} based on a probabilistic reformulation. The strength of the probabilistic approach lies in leveraging the ergodicity property of the microscopic trajectory, which helps to bypass the above-mentioned difficulties when applying a macroscopic PDE method.

More precisely, we construct a stochastic process which can be viewed as a microscopic description of \eqref{eq:FP-PDE}. The probability density of this process gives a weak solution to \eqref{eq:FP-PDE}-\eqref{ic-pde}, analogous to the classical link between a stochastic differential equation and its Fokker-Planck equation. For the process we construct, we will verify the conditions of Harris' theorem, which implies the exponential convergence of its density to the steady state. Our main result is the following. 

\begin{theorem}[Main result]
    For the stochastic process constructed in \eqref{sde-v}-\eqref{ic-sde}, its law $(\mu_t)_{t\geq0}$ is a weak solution to \eqref{eq:FP-PDE}-\eqref{ic-pde} as defined in Definition \ref{def:weak-sol}. Moreover, we have the exponential convergence
   \begin{equation}
       \|\mu_t-\pi\|_{\beta}\leq Ce^{-\lambda t}\|\mu_0-\pi\|_{\beta},\quad t\geq0,
   \end{equation}  to the unique invariant measure $\pi$, where $\|\cdot\|_{\beta}$ is a weighted total variation norm defined in  \eqref{e:norm}, and $C$, $\lambda$ and $\beta$ are positive constants independent of initial data $\mu_0$. 
\end{theorem}

Our probabilistic reformulation is indeed closely related to the stochastic process considered in \cite{cai2006kinetic}, which is the scientific motivation in deriving \eqref{eq:FP-PDE}. However, as far as we know, prior to our work the mathematical analysis on \eqref{eq:FP-PDE} in literature \cite{perthame2013voltage,perthame2018derivation,kim2021fast,dou2022exponential,dpsz2022dcds,CDZ2022} focuses on the PDE side and do not take a probabilistic perspective. Nevertheless, we note that both the PDE and the probabilistic aspects have been studied for some other neuron models, e.g. the NNLIF model \cite{caceres2011analysis,delarue2015global,liu2020rigorous}.

Harris' theorem is a classical tool for convergence of Markov processes \cite{harris1956existence,meyn1994computable,Hairer_Mattingly_Harris}. It gives sufficient conditions to obtain an exponential convergence. In particular, a minorization condition is needed, which roughly means a uniform lower bound for solutions with initial data in a given compact set. Recently, Harris' theorem has been successfully used for many kinetic equations from physics and biology (\cite{canizo2019hypocoercivity,cao2021kinetic,yoldacs2022quantitative,salort2022convergence}), where its conditions are verified by PDE estimates for the time-evolution semigroup. Nevertheless, it seems difficult to verify the minorization condition for \eqref{eq:FP-PDE} via a purely PDE argument, which motivates us to resort to the corresponding stochastic process. Note that exponential convergence for kinetic models has been studied from the stochastic process perspective in earlier works (\cite{mattingly2002ergodicity}). In our case, to obtain the minorization condition, special strategies are designed, making use of the velocity field structure and the jump mechanism of the stochastic process, where the latter is a microscopic reflection of the boundary conditions \eqref{bc-v-2-pde}-\eqref{bc-v-1-pde}.

The rest of this paper is arranged as follows. In Section \ref{sc:2}, we construct a stochastic process which we also establish its link to \eqref{eq:FP-PDE} in Section \ref{sc:3}. Section \ref{sc:4-convergence} is devoted to the proof of the exponential ergodicity for the stochastic process. { A summary and discussions on further direction are given in Section \ref{sc:5-conclusion}.} Appendix~\ref{sc:hormander} briefly introduces H\"{o}rmander's theorem.

To simplify the notations, in the rest of the paper, we fix the parameters
\begin{equation}\label{e:simplify-notations}
    V_R=0,\,V_F = 1,\quad a = g_{\text{in}} = 1.
\end{equation} Such a simplification is not essential for our analysis. Note that we keep $g_L(>0)$ and $V_E(>V_F=1)$ as unspecified parameters.

\section{Probabilistic formulation: an associated process}\label{sc:2}

To show the exponential convergence to the equilibrium of \eqref{eq:FP-PDE}, we first reformulate it from a probabilistic viewpoint. To this end, we construct a renewal type stochastic process whose Fokker-Planck equation is given by \eqref{eq:FP-PDE}-\eqref{ic-pde}. 

We denote the stochastic process by $(V_t,G_t)_{t\geq0}$, where $V_t$ corresponds to the voltage variable $v$ and $G_t$ corresponds to the conductance variable $g$. Its dynamics is constructed as the following stochastic differential equation with renewal in $V_t$.
\begin{align}\label{sde-v}
    dV_t&=J(V_t,G_t)dt,\qquad \text{when }V_t\in[0,1),\\
   \label{sde-bc-v} V_{t-}&=1,\quad \Rightarrow \; \text{set }V_t=0,\\
\label{sde-g}    dG_t&=-(G_t-1)dt+\sqrt{2}dB_t+d L_t,
\end{align}
with initial condition
\begin{equation}\label{ic-sde}
(V_0,G_0)\sim \mu_0,
\end{equation} where $\mu_0$ is a probability measure on the state space $\mathcal{X}$ defined by
\begin{equation}
	\label{e:x}
	    \mathcal{X}:=[0,1)\times[0,+\infty).
\end{equation}
In \eqref{sde-g}, the local time $\{L_t\}_{t\geq0}$ is the minimal non-decreasing process with $L_0=0$ which ensures $G_t\geq0$ for every $t\geq0$. Note that $L_t$ increases at time $t$ only if $G_t=0$, which corresponds to the reflective boundary at the origin. Since the term ``$dL_t$'' in \eqref{sde-g} does not take effect when $G_t>0$, $G_t$ just acts like an ordinary Ornstein-Uhlenbeck (OU) process on $\{t:G_t>0\}$. We conclude that \eqref{sde-g} gives an OU process on $\mathbb{R}^+$ with a reflective boundary condition at the origin (see e.g. \cite{bookMR0192521,ha2009applications}). Note that the evolution of $G_t$ does not depend on $V_t$. 

The dynamics of $V_t$ is governed by two mechanisms. When $V_t\in[0,1)$, \eqref{sde-v} indeed gives an ODE for each realization of trajectory of $G_t$
\begin{equation}\label{ode-v-rewrite}
    \frac{dV_t}{dt}=J(V_t,G_t)=-g_LV_t+G_t(V_E-V_t)\;,\quad \text{when }V_t\in[0,1)\;.
\end{equation}
On the other hand, when the left limit in time of $V_t$ approaches $1$, i.e., $V_{t-}:=\lim_{s\rightarrow t^-}V_s=1$, by \eqref{sde-bc-v}, we set $V_t=0$ and reinitialize the ODE \eqref{ode-v-rewrite}. See Figure~\ref{fig:trajectory} for a typical trajectory of $(V_t,G_t)$.

\begin{figure}[htbp]
    \centering
    \includegraphics[width=0.6\linewidth]{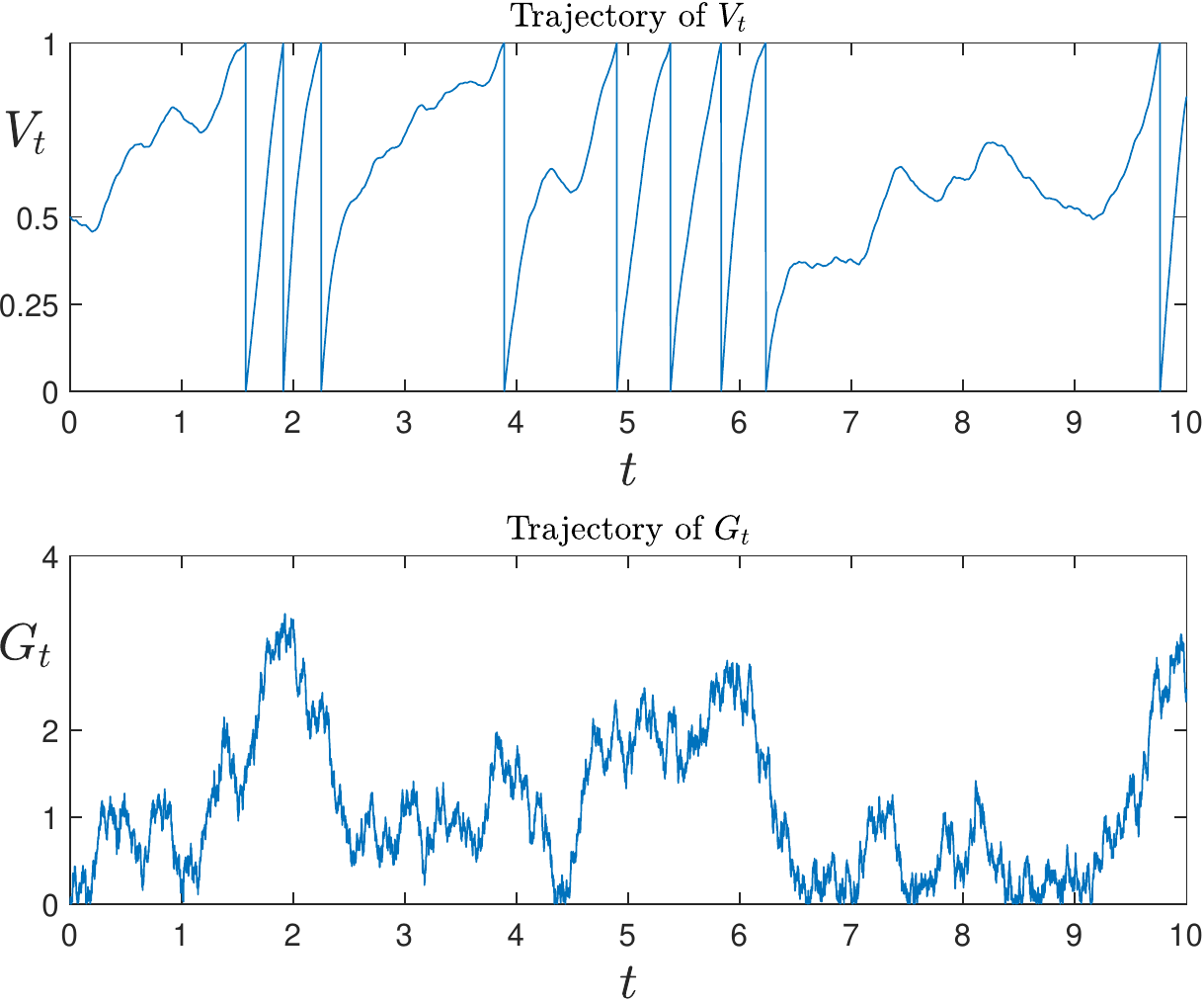}
    \caption{\footnotesize A typical trajectory of $(V_t,G_t)$ via numerical simulations. Parameters: $g_L=1,V_E=2$, other parameters are the same as in \eqref{e:simplify-notations}. }
    \label{fig:trajectory}
\end{figure}

We denote such a time $t$ when $V_t$ has a jump as {a jump time or a spike time}. Biologically, a neuron spikes when its voltage reaches the threshold $V_F=1$. And its voltage is reset to $V_R=0$ after the spike. The ``renewal condition'' \eqref{sde-bc-v} can be viewed as a manifestation of the spike-reset mechanism. Indeed, the SDE system \eqref{sde-v}-\eqref{sde-g} is closely related to the particle system considered in \cite{cai2004effective,cai2006kinetic}, which is the physical motivation to consider the PDE \eqref{eq:FP-PDE}. 

We first give the (pathwise) global well-posedness of \eqref{sde-v}-\eqref{ic-sde} in the following. 

\begin{theorem}\label{thm:sde-existence}
For every initial data $\mu_0$, a probability measure on $\mathcal{X}$, the solution to the SDE system \eqref{sde-v}-\eqref{ic-sde} exists for all time $t\geq0$ and is trajectory-wise unique almost surely. 
\end{theorem}
\begin{proof}
The global well-posedness of $(G_t)_{t \geq 0}$ is immediate since \eqref{sde-g} gives a standard reflected OU process (see e.g. \cite{ha2009applications,bookMR0192521}). Moreover, $G_t$ has a continuous trajectory almost surely. 

It remains to study $(V_t)_{t \geq 0}$. According to the ODE structure, we first solve \eqref{ode-v-rewrite} when $V_t\in[0,1)$, and reset $V_t=0$ whenever $V_{t^-}=1$, and then repeat this procedure. To show global existence, it suffices to show that there are no infinite jump times in any finite time interval, so that we can repeat the above-mentioned piecewise construction forever. This situation is similar to \cite[Lemma 3.3]{delarue2015particle}. Let $\tau_0 = 0$. For $k \geq 1$, let $\tau_k$ be the $k$-th jump time given by
\begin{equation}\label{def:jump-time}
    \tau_k := \inf\Big\{ t> \tau_{k-1} : \lim_{s\rightarrow t^-} V_{s} = 1 \Big\}\;.
\end{equation}
Then it suffices to show
\begin{equation} \label{e:non_explosive}
    \mathbb{P}\left(\,\lim_{k\rightarrow+\infty}\tau_k=+\infty\,\right)=1.
\end{equation}
For fixed $N>0$, if $\tau_k \leq N$ for every $k\geq0$, then by \eqref{sde-v} and that $J(v,g) \leq g V_E$, we have
\begin{equation}\label{tmp-pf-global}
    \tau_k - \tau_{k-1} \geq \frac{1}{\max_{0\leq t\leq N}J(V_t,G_t)} \geq \frac{1}{V_E\max_{0\leq t\leq N}G_t}\;.
\end{equation}
The last quantity in \eqref{tmp-pf-global} tends to $0$ if and only if $\max_{0\leq t\leq N}G_t = + \infty$. Therefore, we obtain
\begin{equation*}
    \mathbb{P}\Big(\lim_{k\rightarrow+\infty}\tau_k \leq N\Big) = \mathbb{P}\Big(\lim_{k\rightarrow+\infty}\tau_k \leq N,\lim_{k\rightarrow+\infty}(\tau_k-\tau_{k-1}) = 0\Big) \leq \mathbb{P}\Big(\max_{0\leq t\leq N}G_t = + \infty\Big) = 0.
\end{equation*}
Then \eqref{e:non_explosive} follows by sending $N \rightarrow +\infty$. 
\end{proof}

Next, we connect the stochastic process \eqref{sde-v}-\eqref{ic-sde} with the PDE \eqref{eq:FP-PDE}-\eqref{ic-pde} in Section~\ref{sc:3}. And we prove the exponential convergence to the steady state for the SDE \eqref{sde-v}-\eqref{ic-sde} in Section~\ref{sc:4-convergence}.

\section{From the process to its Fokker-Planck equation}\label{sc:3}

To justify the probabilistic reformulation introduced in the previous section, we derive \eqref{eq:FP-PDE}-\eqref{ic-pde} as the Fokker-Planck equation of \eqref{sde-v}-\eqref{ic-sde}. More precisely, we shall show the law of $(V_t,G_t)$ gives a weak solution to the PDE \eqref{eq:FP-PDE}-\eqref{ic-pde}. 

In Section \ref{sc:3.1} we introduce our definition of weak solution (Definition~\ref{def:weak-sol}), discuss its basic properties (Propositions~\ref{prop-regularity-weak} and~\ref{prop:bc}), and show its connection with the SDE \eqref{sde-v}-\eqref{ic-sde} (Theorem \ref{thm:sde-pde}). Some of the proofs are postponed to Section \ref{sc:3.2} and \ref{sc:3.3-BC}.

We define the generator $\mathcal{L}$ by
\begin{equation} \label{e:L}
    \mathcal{L}u = \partial_{gg}u + (1-g)\partial_{g}u + J\partial_{v}u\;.
\end{equation}
Its adjoint $\mathcal{L}^*$ is thus given by
\begin{equation*}
    \mathcal{L}^*u = \partial_{gg}u - \partial_{g} ((1-g)u) - \partial_{v}(J u)\;.
\end{equation*}

\subsection{The Fokker-Planck equation: weak formulation}\label{sc:3.1}

The definition of a weak solution to \eqref{eq:FP-PDE}-\eqref{ic-pde} is introduced as follows.

\begin{definition}[Weak solution of \eqref{eq:FP-PDE}-\eqref{ic-pde}]\label{def:weak-sol} Let $(\mu_t)_{t\geq 0}$ be a family of probability measures on $[0,1)\times[0,+\infty)$ indexed by $t \geq 0$. We say $(\mu_t)_{t\geq 0}$ is a weak solution to \eqref{eq:FP-PDE}-\eqref{ic-pde} with initial data $\mu_0$ if the followings hold:
\begin{enumerate}
    \item For each $t>0$, $\mu_t$ is absolutely continuous with respect to the Lebesgue measure on $\mathcal{X}=[0,1)\times[0,+\infty)$.
    \item We have the identity
\begin{equation}  \label{lem:global_ito}
		\int_{\mathcal{X}}\phi(T,v,g)\mu_T(dv,dg) = \int_{\mathcal{X}}\phi(0,v,g)\mu_0(dv,dg) + \int_0^T \int_{\mathcal{X}} (\partial_t + \mathcal{L})\phi(t,v,g)\mu_t(dv,dg) dt,
\end{equation} for every $T>0$ and $\phi \in \mathbf{\Phi}$ . Here the admissible class of test functions $\mathbf{\Phi}$ is defined by
\begin{equation}\label{class-test-function}
    	\mathbf{\Phi}:= \big\{  \phi(t,v,g) \in \mathcal{C}^{\infty}_{c}\big([0,+\infty) \times \mathcal{X} \big)\,: \; \phi(t,0,g) = \phi(t,1,g),\; \partial_g\phi(t,v,0)=0  \big\}\;.
\end{equation}
\end{enumerate}
\end{definition}

\begin{remark}
    The restriction \eqref{class-test-function} on the test function $\phi$ is a reflection of the boundary conditions. We shall see it more clearly in the discussion around Proposition~\ref{prop:bc} and in the proof of Theorem~\ref{thm:sde-pde}. We should notice that the domain $[0,+\infty) \times \mathcal{X}$ for $(t,v,g)$ includes the boundaries $\{t=0\}$ and $\{g=0\}$. Hence, $\phi$ having compact support in this domain does not imply $\phi(t,v,0) = 0$ (and we do not require it). In fact, we need to allow our test functions to have arbitrary values at $g=0$ to ensure the boundary condition \eqref{bc-g-pde}. 
\end{remark}

Let us give two propositions to elucidate Definition \ref{def:weak-sol} before we state its connection to the stochastic process \eqref{sde-v}-\eqref{ic-sde}. We shall see in the first proposition that the second condition \eqref{lem:global_ito} implies \eqref{eq:FP-PDE} in the distributional sense, by restricting to appropriate test functions. Further, the second condition itself also ensures the \textit{interior} regularity of $\mu_t$ thanks to the hypoellipticity, which implies that the equation \eqref{eq:FP-PDE} is indeed satisfied in the classical sense.

\begin{proposition}\label{prop-regularity-weak}
    Let $(\mu_t)_{t\geq 0}$ be a family of probability measures on $[0,1)\times[0,+\infty)$ satisfying \eqref{lem:global_ito}. Then $\mu_t$ is a $\mathcal{C}^{\infty}$ function inside $(0,1)\times(0,+\infty)$. More precisely, there exists a smooth non-negative function $p \in \mathcal{C}^{\infty}\bigl((0,+\infty)\times(0,1)\times(0,+\infty)\bigr)$ such that for every $t > 0$, one has
\begin{equation}\label{interior-density}
    \mu_t(A)=\int_{A}p(t,v,g)dvdg,\qquad \text{for all measurable $A\subseteq (0,1)\times(0,+\infty)$}\;.
\end{equation}
Moreover, the equation \eqref{eq:FP-PDE} is satisfied in the classical sense. 
\end{proposition}
\begin{proof}
We restrict $\phi \in \mathcal{C}^{\infty}_c((0,+\infty)\times(0,1)\times(0,+\infty))\subset \mathbf{\Phi}$ in the weak formulation \eqref{lem:global_ito} to derive
\begin{equation}\label{global-ito-space}
    \int_0^{+\infty} \int_{\mathcal{X}} (\partial_t + \mathcal{L})\phi(t,v,g)\mu_t(dv,dg) dt=0,\qquad \forall \phi \in \mathcal{C}^{\infty}_c((0,+\infty)\times(0,1)\times(0,+\infty))\;.
\end{equation}
This exactly gives that $\mu_t$ satisfies the equation \eqref{eq:FP-PDE} in the distributional sense, that is,
\begin{equation}\label{eq:distributional-sense}
    (\p_t-\mathcal{L}^*)\mu_t=0 \qquad \text{in} \;  \mathcal{D}' \big((0,+\infty)\times(0,1)\times(0,+\infty) \big).
\end{equation}
Recall an operator $K$ is hypoelliptic in a domain $\Omega$ if $Ku\in \mathcal{C}^{\infty}(U)$ implies $u\in \mathcal{C}^{\infty}(U)$ for every open set $U\subseteq\Omega$. By H\"ormander's Theorem, $\p_t-\mathcal{L}^*$ is hypoelliptic (see Lemma \ref{lem:hypoelliptic} for more details). Thus by \eqref{eq:distributional-sense} we deduce that $\mu_t$ is indeed has a smooth interior density $p$. Combining this with \eqref{eq:distributional-sense}, we derive that \eqref{eq:FP-PDE} is indeed satisfied in the classical sense. 
\end{proof}

In view of the interior regularity given in Proposition \ref{prop-regularity-weak}, the main point of the first condition in Definition \ref{def:weak-sol} is to ensure the \textit{boundary} regularity, that is, $\mu_t$ is not singular at $v=0$ or $g=0$.

Although the boundary conditions \eqref{bc-g-pde}, \eqref{bc-v-2-pde} and \eqref{bc-v-1-pde} do not show up in Definition \ref{def:weak-sol} explicitly, they are indeed encoded in the class of test functions \eqref{class-test-function}. Formally they arise if we integrate by parts for the last term in \eqref{lem:global_ito}, and use the boundary constraints for test functions in \eqref{class-test-function}. In particular, the first constraint $\phi(t,0,g) = \phi(t,1,g)$ corresponds to the boundary condition at $v=0$, and the second constraint $\partial_g\phi(t,v,0)=0$ corresponds to the boundary condition at $g=0$. For a rigorous statement, we have the following.

\begin{proposition}\label{prop:bc}
  Let $(\mu_t)_{t\geq 0}$ be a weak solution to  \eqref{eq:FP-PDE}-\eqref{ic-pde} as in Definition \ref{def:weak-sol}. Denote its interior density function as $p(t,v,g)$. Then the boundary conditions in $v$ \eqref{bc-v-2-pde}-\eqref{bc-v-1-pde} are satisfied in the following sense:
    for every test function $\psi(t,g)\in \mathcal{C}^{\infty}_c\bigl((0,+\infty)\times(0,+\infty)\bigr)$, we have
    \begin{equation}\label{bc-v-weak}
\lim_{v\rightarrow0^+}\int_{\mathbb{R}^+\times\mathbb{R}^+}J(v,g)p(t,v,g)\psi(t,g)dtdg=\lim_{v\rightarrow1^-}\int_{\mathbb{R}^+\times\mathbb{R}^+}J(v,g)p(t,v,g)\psi(t,g)dtdg,
    \end{equation}
    where both limits exist and are finite. If in addition $\psi \geq 0$ with support in $(0,+\infty)\times(0,g_F)$, then both limits in \eqref{bc-v-weak} are $0$. 
\end{proposition}
Here we only state for the $v$ boundary condition as it is a  unique feature of this problem, and the $g$ direction is similar. We postpone the proof of Proposition \ref{prop:bc} to Section \ref{sc:3.3-BC}. 

Note that Proposition~\ref{prop:bc} only assumes $(\mu_t)_{t \geq 0}$ satisfying Definition~\ref{def:weak-sol}, and Proposition~\ref{prop-regularity-weak} only assumes $(\mu_t)_{t \geq 0}$ satisfying \eqref{lem:global_ito}. In particular, they do not assume any relationship between $(\mu_t)_{t \geq 0}$ and the process $(V_t, G_t)$. 

In what follows, we will turn to the process $(V_t, G_t)$. We first have the following proposition concerning its boundary behavior. Its proof is postponed to Section \ref{sc:3.2}.

\begin{proposition}\label{prop:boundary-regularity}
Let $(V_t,G_t)_{t\geq0}$ be the solution of \eqref{sde-v}-\eqref{ic-sde}. Then we have
\begin{equation}\label{regular-v=0}
    \mathbb{P}\left({V_t=0}\right)=\mathbb{P}\left(V_{t^-}=1\right)=0,\qquad t>0,
\end{equation} and
\begin{equation}\label{regular-g=0}
    \mathbb{P}\left(G_t=0\right)=0,\qquad t>0.
\end{equation}
\end{proposition}

We are now ready to establish the relation between the stochastic process \eqref{sde-v}-\eqref{ic-sde} and the PDE \eqref{eq:FP-PDE}-\eqref{ic-pde}. This is the content of the following theorem. 

\begin{theorem}\label{thm:sde-pde}
Let $(V_t,G_t)_{t\geq0}$ be the unique solution to the SDE \eqref{sde-v}-\eqref{ic-sde} established in Theorem \ref{thm:sde-existence}. Denote the law of $(V_t,G_t)$ as $\mu_t$. Then $(\mu_t)_{t\geq0}$ is a weak solution to \eqref{eq:FP-PDE}-\eqref{ic-pde} with initial data $\mu_0$ in the sense of Definition~\ref{def:weak-sol}.
\end{theorem}
\begin{proof}
    \textit{Part 1: weak formulation.} First, we show that $\mu_t$ satisfies the weak formulation \eqref{lem:global_ito}, which is the second condition in Definition \ref{def:weak-sol}. For $\phi(t,v,g)\in \mathcal{C}^{\infty}_c(\mathbb{R}^3)$, by It\^o's formula for semimartingales with jumps (see for example \cite[Theorem~2.33]{protter}), we get 
	\begin{equation}\label{Ito-semi}
	\begin{split}
		\phi(T,V_T,G_T)  - \phi(0, &V_0,G_0) =\int_0^T (\partial_t + \mathcal{L})\phi(t,V_{t-},G_t) dt+ \sqrt{2}\int_0^T \partial_g \phi(t,V_{t-},G_t)dB_t \\
        &+ \int_0^T \partial_g\phi(t,V_{t-},G_t) dL_t+ \sum_{k\geq 1:\;\tau_k\leq T} \big(\phi(\tau_k,0,G_{\tau_k}) - \phi(\tau_k,1,G_{\tau_k})\big)\,
	\end{split}
	\end{equation} 
     where $\tau_k$ is the $k$-th jump time in $v$ direction defined in \eqref{def:jump-time}. Note that the first two terms on the right hand side of \eqref{Ito-semi} appear in the usual It\^o's formula, the term of the integration with respect to the local time $L_t$ appears since the reflective boundary in $g$-direction, and the last term appears since $V_t$ jumps from $1$ to $0$ at the boundary in $v$-direction. 
 
	Then, we restrict $\phi\in\mathbf{\Phi}$ defined in \eqref{class-test-function} to get rid of the latter two terms in \eqref{Ito-semi} and derive
	\begin{equation}\label{weak-pathwise}
		\phi(T,V_T,G_T)  - \phi(0,V_0,G_0) = \int_0^T (\partial_t + \mathcal{L})\phi(t,V_{t},G_t) dt + \sqrt{2}\int_0^T \partial_g \phi(t,V_{t-},G_t)dB_t 
	\end{equation}
    Indeed for $\phi\in\mathbf{\Phi}$, the sums of the jumping terms vanish since $\phi(t,0,g) = \phi(t,1,g)$, and the integration with respect to local time process $L_t$ also vanishes since $L_t$ changes only when $G_t=0$ but $(\partial_g \phi)(t,v,0) = 0$ for $\phi \in \mathbf{\Phi}$. Taking expectation on both sides in \eqref{weak-pathwise} gives \eqref{lem:global_ito}. 

    \textit{Part 2: regularity of $\mu_t$.} We now confirm the first condition in Definition~\ref{def:weak-sol}. Since $\mu_t$ satisfies \eqref{lem:global_ito}, we can apply Proposition~\ref{prop-regularity-weak} to obtain the interior regularity of $\mu_t$ given in \eqref{interior-density}. Combining the interior regularity and Proposition~\ref{prop:boundary-regularity} above, we immediately deduce that $\mu_t$ is absolutely continuous with respect to the Lebesgue measure on $[0,1)\times[0,+\infty)$.
\end{proof}

The rest of this section is arranged as allows. Section \ref{sc:3.2} is devoted to proving Proposition \ref{prop:boundary-regularity}, which completes the proof of Theorem \ref{thm:sde-pde}.  In Section \ref{sc:3.3-BC} we start from Definition \ref{def:weak-sol} to explicitly derive  the boundary conditions, i.e. proving Proposition \ref{prop:bc}.

\subsection{Boundary regularity: Proof of Proposition \ref{prop:boundary-regularity}}\label{sc:3.2}

This section is devoted to proving the boundary regularity in Proposition \ref{prop:boundary-regularity}. We give details for \eqref{regular-v=0} only since \eqref{regular-g=0} is a standard property for the reflected OU process (see e.g. \cite{bookMR0192521,ha2009applications}). 

For the interior regularity (Proposition~\ref{prop-regularity-weak}), we used the hypoellipticity of the operator $\p_t-\mathcal{L}^*$ thanks to H\"ormander's theorem. Since hypoellipticity is a local property, a direct use of H\"ormander's theorem does not imply the boundary behavior. To prove the boundary regularity in Proposition~\ref{prop:boundary-regularity}, an auxiliary process is introduced in \eqref{sde-without-jump} which extends the domain of $v$ beyond $(0,1)$. 

Consider the Markov process $\widehat{V}_t$ without jumps, given by
\begin{equation}\label{sde-without-jump}
	d\widehat{V}_t = J(\widehat{V}_t,G_t)dt,\quad \widehat{V}_0 = V_0,\qquad \widehat{V}_t\in\mathbb{R}.
\end{equation}
Note that $J(v,g)=-g_Lv+g(V_E-v)$ is naturally defined for all $v\in\mathbb{R}$. The process \eqref{sde-without-jump} is just $V_t$ in \eqref{sde-v} without the renewal condition \eqref{sde-bc-v}. Because the $v$-domain is extended, $v=1$ becomes an interior point for $\hat{V}_t$, thus allowing us to apply H\"ormander's theorem.

We first give the smoothness of the density of $(\widehat{V}_t,G_t)$ in Lemma~\ref{lem:density_continuity_1} by the hypoellipticity property given in Lemma~\ref{lem:hypoelliptic}.

\begin{lemma} \label{lem:density_continuity_1}
	For every $(v_0,g_0)\in\mathcal{X}$, the distribution of $(\widehat{V}_t,G_t)$ with initial data $(\widehat{V}_0,G_0)=(v_0,g_0)$ has a density $\widehat{\rho}_t^{v_0,g_0}(v,g)$ for $t>0$, and the mapping
	\begin{equation*}
		(t,v_0,g_0,v,g) \mapsto \widehat{\rho}_t^{v_0,g_0}(v,g) \in \mathcal{C}^{\infty}(\mathbb{R}^+\times(0,V_E)\times\mathbb{R}^+\times(0,V_E)\times\mathbb{R}^+).
	\end{equation*}
\end{lemma}

\begin{proof}
	The proof is essentially the same as \cite[Theorem~4.3]{foldes2021method}. The difference is that the state space of $(\widehat{V}_t,G_t)$ has a reflective boundary in $g$-direction. We only provide details to the continuity with respect to $(t,v,g)$ to demonstrate the modified calculation on the reflected OU process. 
 
    For every $T>0$ and every smooth function $\phi$ with $\partial_g\phi(t,v,0) = 0$ and with compact support in $\mathcal{C}^{\infty}_c((0,T)\times\mathbb{R}\times\mathbb{R}^+)$ , by It\^o's formula, we have
	\begin{equation} \label{e:ito-modified}
		0=\phi(T,\widehat{V}_T,G_T) =  \int_0^T(\partial_t + \mathcal{L})\phi(t,\widehat{V}_t,G_t) dt+ \int_0^T \partial_{g}\phi(t,\widehat{V}_t,G_t) d L_{t} + \sqrt{2} \int_0^T \partial_{g}\phi(t,\widehat{V}_t,G_t) d B_t.
	\end{equation}
	Again, note that the local time $L_t$ only increases when $G_t = 0$, the term with the local time process vanishes almost surely. Taking expectations on both sides of \eqref{e:ito-modified} and integrating by parts, we have
	\begin{equation*}
		\langle  \phi(t,v,g), (-\partial_t+\mathcal{L}^*) \widehat{\rho}^{v_0,g_0}_t(v,g) \rangle_{t,v,g} = 0.
	\end{equation*}
	By the choice of $\phi$,
	\begin{equation*}
		(\partial_t - \mathcal{L}^*) \widehat{\rho}^{v_0,g_0}_{t} = 0
	\end{equation*}
	holds in the domain $(0,T)\times\mathbb{R}\times\mathbb{R}^+$ in the distribution sense. By Lemma~\ref{lem:hypoelliptic}, $\partial_t - \mathcal{L}^*$ is hypoelliptic in $(0,T)\times(0,V_E)\times\mathbb{R}^ +$, which implies that the mapping $(t,v,g)\mapsto\widehat{\rho}^{v_0,g_0}_t(v,g)$ is smooth in the same domain. Therefore, the desired result holds since $T>0$ is arbitrary.

    Similarly, the continuity with respect to the initial data $(v_0,g_0)$ follows from the Kolmogorov backward equation in the distribution sense and the hypoellipticity of $\partial_t+\mathcal{L}$ given by Lemma~\ref{lem:hypoelliptic}. For the proof of the Kolmogorov backward equation, we can modify the details in \cite[Theorem~4.3]{foldes2021method} as above to treat the local time term caused by the reflection boundary in $g$-direction.
\end{proof}

Note that the SDE of $\widehat{V}_t$ has the same form as \eqref{sde-v} except that $V_t$ jumps to $0$ when approaching $1$. The two processes are the same before the first jump time $\tau_1$ of $V_t$ defined in \eqref{def:jump-time}. Hence we have
\begin{equation} \label{e:v_hat=v_tilde}
	\widehat{V}_{(t\wedge\tau_1)^-} = V_{(t\wedge\tau_1)^-}.
\end{equation}
Now we are going to prove Proposition \ref{prop:boundary-regularity}. As mentioned, the $g$-boundary regularity \eqref{regular-g=0} is standard for the reflected OU process (see e.g. \cite{ha2009applications,bookMR0192521}). It remains to show the $v$-boundary regularity \eqref{regular-v=0}.

\begin{proof}[Proof of Proposition \ref{prop:boundary-regularity}]
    We first note that for $t>0$, the event $\{V_t = 0\}$ is the disjoint union of the two events $\{V_{t^{-}}=1\}$ and $\{V_{t^{-}} = 0\}$. Hence, we have
    \begin{equation}\label{0-VF-tmp}
        \mathbb{P}\left({V_t=0}\right)=\mathbb{P}\left(V_{t^-}=1\right)+\mathbb{P}(V_{t^-}=0),\qquad t>0.
    \end{equation}
    If $V_{t^{-}}=0$, we can deduce that the evolution of $V$ follows the ODE \eqref{sde-v} in a left neighborhood of $t$. This implies $G_t=0$ since $J(0,g)> 0$ for every $g>0$. Hence $\{V_{t^{-}}=0\}\subseteq\{G_t=0\}$ which has probability zero due to \eqref{regular-g=0}. By \eqref{0-VF-tmp}, we get
    \begin{equation*}
	   \mathbb{P}\left(V_t = 0\right)= \mathbb{P}\left(V_{t^-}=1\right)= \sum_{k=1}^{\infty} \mathbb{P}(\tau_k = t),
    \end{equation*}
    where $\tau_k$ is the $k$-th jump time defined in \eqref{def:jump-time}. It remains to prove $\mathbb{P}(\tau_k = t)=0$ for every $k\in\mathbb{N}^+$.
    
    We first consider $k=1$. Since $\widehat{V}_t$ has the same trajectory as $V_t$ before $\tau_1$, by \eqref{e:v_hat=v_tilde}, we have
    \begin{equation*}
	   \mathbb{P}(\tau_1 = t) \leq \mathbb{P}(\widehat{V}_{t} = 1)\;.
    \end{equation*}
    Since $v=1$ is a interior point of $(0,V_E)$, Lemma~\ref{lem:density_continuity_1} implies
    \begin{equation*}
	   \mathbb{P}(\widehat{V}_{t} = 1) = \int_{\mathcal{X}} \mathbb{P}^{v_0,g_0}(\widehat{V}_{t} = 1) \mu_0(dv_0,dg_0) = 0,
    \end{equation*}
    where $\mathbb{P}^{v_0,g_0}$ is the probability measure conditioned on $(V_0,G_0)=(v_0,g_0)$. Hence, $\mathbb{P}(\tau_1=t) = 0$.
 
    Now we turn to $k\geq2$. Denote the distribution of $(\tau_1, G_{\tau_1})$ conditioned on $\{V_0=v,G_0=g\}$ by $\mathcal{T}^{v,g}$, which describes the joint distribution of the first jump time and location. By strong Markov property, the joint law of $(\tau_{j+1}-\tau_j,G_{\tau_{j+1}})$ conditioned on $(\tau_j, G_{\tau_j})$ does not depend on the value of $\tau_j$, and is exactly $\mathcal{T}^{0,G_{\tau_j}}$. Now we compute the probability $\mathbb{P}(\tau_k=t)$ by cutting the trajectory at time $\tau_1$. More precisely, we have
    \begin{equation*}
    \begin{split}
		\mathbb{P}(\tau_k = t) =&\int_{\mathcal{X}} \mathbb{P}^{v_0,g_0}(\tau_k=t) \mu_0(dv_0,dg_0)\\
        =&\int_{\mathcal{X}} \int_{0<t_1<t} \int_{g_1>0} \mathbb{P}(\tau_k-\tau_1=t-t_1|G_{\tau_1}=g_1) \mu_0(dv_0,dg_0) \mathcal{T}^{v_0,g_0}(dt_1,dg_1).
    \end{split}
    \end{equation*}
    Proceeding this procedure, we cut the trajectory at the time $\tau_j$ for $2 \leq j \leq k-1$ successively to get
    \begin{equation*}
		\begin{split}
        \mathbb{P}(\tau_k = t) =& \int_{\mathcal{X}}\int_{0<\sum_{j=1}^{k-1}t_j < t} \int_{\vec{g}\in(\mathbb{R}^+)^{k-1}} \mathbb{P}\bigg(\tau_k-\tau_{k-1} = t- \sum_{j=1}^{k-1}t_j \bigg| G_{\tau_{k-1}}=g_{k-1}\bigg)\\
		&  \qquad\qquad\mu_0(dv_0,dg_0) \mathcal{T}^{v_0,g_0}(dt_{1},dg_{1})\prod_{j=2}^{k-1} \mathcal{T}^{0,g_{j-1}}(dt_{j},dg_{j}),
		\end{split}
	\end{equation*} 
	where $\vec{g} = (g_1,g_2,\dots,g_{k-1})$. By strong Markov property and that $\mathbb{P}(\tau_1=t) = 0$, we get
	\begin{equation*}
		\mathbb{P}\bigg(\tau_k-\tau_{k-1} = t- \sum_{j=1}^{k-1}t_j \bigg| G_{\tau_{k-1}}=g_{k-1}\bigg) = \mathbb{P}^{0,g_{k-1}}\bigg(\tau_1 = t- \sum_{j=1}^{k-1}t_j \bigg) = 0
	\end{equation*}
	for $t-\sum_{j=1}^{k-1}t_j>0$. Therefore, we obtain $\mathbb{P}(\tau_k = t) = 0$ for every $k\in\mathbb{N}^+$.
\end{proof}

\subsection{Deriving the boundary conditions: Proof of Proposition \ref{prop:bc}}\label{sc:3.3-BC}

We have shown that the law of $(V_t,G_t)$ gives a weak solution to \eqref{eq:FP-PDE}-\eqref{ic-pde} defined in Definition \ref{def:weak-sol}. Now we start from a weak solution to derive explicitly the boundary conditions in the sense of Proposition \ref{prop:bc}. As a preparation, we need the following consequence from Definition~\ref{def:weak-sol}. In particular, it uses that $\mu_t$ is not singular at the boundaries $v=0$ and $v=1$.

\begin{lemma} \label{lem:local-bd-regularity}
Let $(\mu_t)_{t\geq0}$ be a weak solution to \eqref{eq:FP-PDE} and \eqref{bc-g-pde}-\eqref{ic-pde} in the sense of Definition \ref{def:weak-sol}. For $t>0$, denote the density of $\mu_t$ as $p(t,v,g)$. Let
\begin{equation*}
q(t,g):=\int_{0}^{1}p(t,v,g)dv=\lim_{\eps\rightarrow0^+}\int_{\eps}^{1-\eps}p(t,v,g)dv,\quad t>0,\, g>0.
\end{equation*}
be the marginal density in $g$. Then, for all $\psi(t,g)\in \mathcal{C}^{\infty}_c(\mathbb{R}^+\times\mathbb{R}^+)$, we have
\begin{equation}\label{FP-g-weak}
    	\int_0^{+\infty} \int_{0}^{+\infty} \big( \partial_t -(g-1)\p_g+ \p_{gg} \big) \psi(t,g) q(t,g) dt dg = 0.
\end{equation}
In other words, $q$ satisfies the equation
\begin{equation}\label{FP-g}
    \p_tq=\p_g((g-1)q)+\p_{gg}q\;, \quad (t,g) \in \mathbb{R}^{+} \times \mathbb{R}^{+}
\end{equation}
in the weak sense, which is the Fokker-Planck equation for the reflected OU process $G$. 
\end{lemma}
\begin{proof}
By Definition~\ref{def:weak-sol}, for every $t>0$, $\mu_t$ is absolutely continuous with respect to the Lebesgue measure with density $p$. Therefore, we can rewrite \eqref{lem:global_ito} as
\begin{equation}\label{global_ito-density}
    \int_{\mathcal{X}}\phi(T,v,g)p(T,v,g)dvdg = \int_{\mathcal{X}}\phi(0,v,g)\mu_0(dv,dg) + \int_0^T \int_{\mathcal{X}} (\partial_t + \mathcal{L})\phi(t,v,g)p(t,v,g) dtdvdg.
\end{equation}
Note that for $\psi: (t,g) \mapsto \psi(t,g)$ in $C_{c}^{\infty}(\mathbb{R}^{+} \times \mathbb{R}^{+})$, the function $\phi(t,v,g):=\psi(t,g)$ belongs to the class of test function $\Phi$ defined in \eqref{class-test-function}. Moreover, as $\psi$ has compact support in $\mathbb{R}^+\times\mathbb{R}^+$, we can choose $T$ large enough such that the first two terms in \eqref{global_ito-density} both vanish, which gives
\begin{equation*}
    \int_0^T \int_{\mathcal{X}} (\partial_t -(g-1)\p_g+ \p_{gg})\psi(t,g)p(t,v,g) dtdvdg.
\end{equation*}
The result then follows from integrating out the $v$ variable.
\end{proof}

Now we begin the proof of Proposition \ref{prop:bc}. 

\begin{proof}[Proof of Proposition \ref{prop:bc}]  
By Proposition \ref{prop-regularity-weak}, the solution $p$ to \eqref{eq:FP-PDE} is smooth in $(0,+\infty) \times (0,t) \times (0,\infty)$. Hence, we can integrate \eqref{eq:FP-PDE} for $v\in(\eps_1,1-\eps_2)\subset(0,1)$ to get 
\begin{equation} \label{integrate-v}
    \begin{split}
            \phantom{11}&\p_t\left(\int_{\eps_1}^{1-\eps_2}p(t,v,g)dv\right)+ \big( J(v,g)p(t,v,g) \big) \big|_{v=\eps_1}^{v=1-\eps_2}\\
            &=\p_g\left((g-1)\int_{\eps_1}^{1-\eps_2}p(t,v,g)dv\right)+\p_{gg}\left(\int_{\eps_1}^{1-\eps_2}p(t,v,g)dv\right), \quad t>0, g>0.
    \end{split}
\end{equation}
We now take the limit $\eps_1,\eps_2\rightarrow0^+$. Formally, in view of Lemma~\ref{lem:local-bd-regularity} and the boundary condition \eqref{bc-v-2-pde}, we expect (the two sides of) \eqref{integrate-v} to converge to those of \eqref{FP-g}. 

To justify this, we need to move the derivatives to test functions, and the boundary condition is obtained in a weak sense. More precisely, multiply \eqref{integrate-v} with a test function $\psi(t,g)\in \mathcal{C}^{\infty}_c(\mathbb{R}^+\times\mathbb{R}^+)$, and integrate by parts for $(t,g)\in\mathbb{R}^+\times\mathbb{R}^+$. Then we have
\begin{equation}\label{flux-integrate}
    \int_0^{+\infty}\int_0^{+\infty} \big(J(v,g)p(t,v,g) \big) \big|_{v=\eps_1}^{v=1-\eps_2}\psi(t,g)dtdg=B_{\psi}(\eps_1,\eps_2),
\end{equation}
where
\begin{equation}\label{Bpsi}
    \begin{aligned}
        B_{\psi}(\eps_1,\eps_2)=    \int_0^{+\infty}\int_0^{+\infty}[(\p_t+\p_{gg}-(g-1)\p_g)\psi(t,g)] \left(\int_{\eps_1}^{1-\eps_2}p(t,v,g)dv\right)dtdg.
    \end{aligned}
\end{equation}
Consider the limit $\eps_1,\eps_2\rightarrow0^+$ in \eqref{Bpsi}. By dominated convergence and Lemma~\ref{lem:local-bd-regularity}, we get
\begin{equation*}
  \lim_{\eps_1,\eps_2\rightarrow0^+}B_{\psi}(\eps_1,\eps_2)=\int_0^{+\infty}\int_0^{+\infty}[(\p_t+\p_{gg}-(g-1)\p_g)\psi(t,g)] q(t,g)dtdg=0,
\end{equation*}
where in the last equality we use that $q$ satisfies \eqref{FP-g-weak}. Hence, in view of \eqref{flux-integrate}, we deduce
\begin{equation}\label{flux-integrated-limit}
    \lim_{\eps_1,\eps_2\rightarrow0^+}\int_{\mathbb{R}^+\times\mathbb{R}^+}\left[ \big(J(v,g)p(t,v,g) \big)|_{v=\eps_1}^{v=1-\eps_2}\right]\psi(t,g)dtdg=0,\quad\quad \forall \psi\in \mathcal{C}^{\infty}_c(\mathbb{R}^+\times\mathbb{R}^+).
\end{equation}
We claim that, as a consequence of \eqref{flux-integrated-limit}, for every $\psi\in \mathcal{C}^{\infty}_c(\mathbb{R}^+\times\mathbb{R}^+)$, the two limits in \eqref{bc-v-weak} exist and are equal. Indeed, for fixed $\psi$, denote
\begin{align*}
    C_{1}(\eps_1)&:=\int_{{\mathbb{R}^+\times\mathbb{R}^+}}J(\eps_1,g)p(t,\eps_1,g)\psi(t,g)dtdg,\\
    C_{2}(\eps_2)&:=\int_{{\mathbb{R}^+\times\mathbb{R}^+}}J(1-\eps_1,g)p(t,1-\eps_2,g)\psi(t,g)dtdg.
\end{align*} Then \eqref{flux-integrated-limit} reads
\begin{equation}\label{flux-integrated-limit-c1c2}
    \lim_{\eps_1,\eps_2\rightarrow0^+}\left(C_{2}(\eps_2)-C_{1}(\eps_1)\right)=0.
\end{equation}
Since the above limit exists as $\eps_1, \eps_2 \rightarrow 0$ in arbitrary ways, we deduce that both $C_1(\eps_1)$ and $C_2(\eps_2)$ converge to finite numbers as $\eps_1, \eps_2 \rightarrow 0$, and their limits agree. This proves \eqref{bc-v-weak}.

Finally, suppose the support of $\psi(t,g)$ is in $(0,+\infty)\times(0,g_F)$ and $\psi\geq 0$. Recall $g_F$ is the unique zero of $J(1,g)$ on $(0,+\infty)$. Then, noting that $J(0,g)>0$ and $J(1,g)<0$ for $g\in(0,g_F)$, we derive from \eqref{bc-v-weak} that
\begin{equation}\label{flux-integrated-limit-tmp-zero}
0\leq \lim_{v\rightarrow0^+}\int_{\mathbb{R}^+\times\mathbb{R}^+}J(v,g)p(t,v,g)\psi(t,g)dtdg=\lim_{v\rightarrow1^-}\int_{\mathbb{R}^+\times\mathbb{R}^+}J(v,g)p(t,v,g)\psi(t,g)dtdg\leq 0,
\end{equation} since $p\geq 0$ and $\psi\geq 0$. Hence the equalities in \eqref{flux-integrated-limit-tmp-zero} hold, which implies both limits are zero. This proves Proposition~\ref{prop:bc}.
\end{proof}

\section{Exponential convergence}\label{sc:4-convergence}

In this section, we will show the exponential ergodicity of the Markov process $(V_t, G_t)$. By Harris' theorem, the exponential convergence follows from the Lyapunov function structure and the minorization condition. The following is a version of Harris theorem given in \cite[Theorem~1.3]{Hairer_Mattingly_Harris}.

\begin{theorem} \label{thm:general_harris}
    Let $\mathcal{P}$ be the Markov operator of a discrete-time Markov process on a measurable space $\mathcal{X}$, and $\mathcal{P}^*$ denotes its adjoint. Suppose the following two assumptions hold.
    \begin{enumerate}
    \item (Lyapunov function)
        There exists $W: \mathcal{X} \mapsto [0,+\infty)$ and $\alpha_1\in(0,1), \alpha_2>0$ such that
        \begin{equation} \label{e:Lyapunov}
		  (\mathcal{P}W)(x) \leq \alpha_1 W(x) + \alpha_2
        \end{equation}
        for every $x\in\mathcal{X}$.
        \item (Minorization condition)
        There exist $\eta>0$, $R > \frac{2 \alpha_2}{1-\alpha_1}$ and a probability measure $\nu$ on $\mathcal{X}$ such that
        \begin{equation}
		  \inf_{x\in \mathcal{C}(R)} \mathcal{P}^* \delta_{x} \geq \eta \nu\;,
        \end{equation}
        where $\mathcal{C}(R):= \{x: W(x)\leq R\}$ and $\delta_{x}$ is the Dirac measure at the point $x$. 
    \end{enumerate}
    Then there exists a unique stationary distribution $\pi$ with respect to the semi-group $\mathcal{P}$. Furthermore, there exist $\beta>0$ and $\theta \in (0,1)$ such that we have the exponential convergence
    \begin{equation*} 
		\| (\mathcal{P}^*)^n \mu  - \pi \|_{\beta} \leq  \theta^{n} \| \mu - \pi \|_{\beta}
	\end{equation*}
    for every initial distribution $\mu$ and every $n \geq 1$, where
    \begin{equation} \label{e:norm}
	   \| \nu_1-\nu_2 \|_{\beta} := \int_{\mathcal{X}} (1+\beta W(x))|\nu_1-\nu_2|(dx).
    \end{equation}
    is the total variation distance norm weighted by the Lyapunov function $W$ and constant $\beta > 0$. 
\end{theorem}

Applying Harris' theorem to our model, we obtain the main result on the exponential ergodicity of the process $(V_t,G_t)$.

\begin{theorem}
\label{thm:harris}
    Let $(\pP_t)_{t \geq 0}$ be the Markov semigroup for the process $(V_t, G_t)$, and $(\pP_t^*)_{t \geq 0}$ be its adjoint. Then, there exists a unique stationary distribution $\pi$ with respect to $(\pP_t)_{t \geq 0}$. Furthermore, there exist $\beta, C, \lambda > 0$ such that
    \begin{equation} \label{e:harris}
        \|\pP_t^* \mu_0 - \pi\|_{\beta} \leq C e^{-\lambda t} \|\mu_0 - \pi\|_{\beta}
    \end{equation}
    for every initial distribution $\mu_0$ on $\mathcal{X}$ and every $t>0$, where $\|\cdot\|_{\beta}$ is the weighted total variation distance norm defined in \eqref{e:norm}. 
\end{theorem}

The exponential decay \eqref{e:harris} will first be proved along a subsequence $t_n = n T$ for some fixed $T>0$ (to be specified later in Proposition~\ref{lem:minorisation}), and then extended to all $t>0$. To prove it along such a subsequence, we will verify the assumptions in Theorem~\ref{thm:general_harris} for the Markov operator $\mathcal{P}_{T}$.

\subsection{Lyapunov function structure} 

Since the domain $\mathcal{X}$ given by \eqref{e:x} is bounded in $v$-direction, the Lyapunov function can be chosen to depend on $g$ only.

\begin{lemma} \label{lem:Lyapunov}
	For every $T>0$, the function $W(v,g):=(g-1)^2$ {is a Lyapunov function with respect to the Markov operator $\mathcal{P}_T$, as it satisfies Assumption~1 in Theorem~\ref{thm:general_harris} }.
\end{lemma}

\begin{proof}
Fix $(v_0,g_0)\in\mathcal{X}$. Define
    \begin{equation*}
        f(t):=\mathbb{E}^{(v_0,g_0)}W(V_t,G_t)=\mathbb{E}^{(v_0,g_0)}(G_t-1)^2.
    \end{equation*}
Recall It\^o's formula for semi-martingales given in \eqref{Ito-semi}. Taking $\phi(t,v,g) = W(v,g) =(g-1)^2$, we obtain
\begin{equation*}
    \frac{1}{2} \big((G_t-1)^2 - (G_0-1)^2 \big) = \int_0^t \big( 1-(G_s-1)^2 \big) ds + \sqrt{2} \int_0^t (G_s-1) dB_s + \int_0^t (G_s-1) dL_s.
\end{equation*}
Taking expectations on both sides, we get
\begin{equation} \label{e:ito_semi_lyapunov}
    \frac{1}{2} \big(f(t) - f(0) \big) = t - \int_0^t f(s) ds + \mathbb{E}^{(v_0,g_0)}\bigg[\int_0^t (G_s-1) dL_s\bigg].
\end{equation}
Recall that $L_s$ changes its value only when $G_s=0$, so we have
\begin{equation*}
    \int_{0}^{t} (G_s - 1) d L_s = - L_t\;.
\end{equation*}
Differentiating with respect to $t$ and using that $L$ is increasing, we get
\begin{equation*}
    f'(t) \leq 2 - 2 f(t)\;.
\end{equation*}
Hence by Gronwall's inequality, we conclude 
\begin{equation} \label{e:lyapunov}
    \mathbb{E}^{(v_0,g_0)}W(V_t,G_t)=f(t) \leq e^{-2t}f(0)+(1-e^{-2t}) =e^{-2t}W(v_0,g_0)+(1-e^{-2t}),
\end{equation}
which verifies \eqref{e:Lyapunov}.
\end{proof}

\subsection{Minorization condition}

Now we are going to check Assumption~2 in Theorem~\ref{thm:general_harris}. Recall that we choose $W(v,g):=(g-1)^2\rightarrow+\infty$ as $g\rightarrow+\infty$., then for $R\geq 1$, we have
\begin{equation} \label{e:c}
    \mathcal{C}(R) = \{(v,g):W(v,g)\leq R\} = [0,1)\times[0,M(R)],
\end{equation}
where $M(R):=\sqrt{R}+1$. The minorization condition is a direct corollary of the following proposition. 
\begin{proposition}
	\label{lem:minorisation}
	There exists $T>0$ such that for every $R\geq1$, there exists a constant $\eta(R)>0$ and a probability measure $\nu$ on $\mathcal{X}$ such that
	\begin{equation} \label{e:lower_bound}
		\inf_{(v_0,g_0)\in \mathcal{C}} \mathbb{P}^{v_0,g_0}((V_T,G_T)\in A) \geq \eta \nu(A)
	\end{equation}
	for every Borel measurable $A\subset\mathcal{X}$, {where $\mathcal{C}$ is given in \eqref{e:c}}.
\end{proposition}

Figure~\ref{fig:strategy} demonstrates the strategy of proving Proposition~\ref{lem:minorisation}. We decompose the Markov process $\{(V_t,G_t)\}_{0\leq t\leq T}$ into two parts: $t \in [0,T_1]$ and $t \in [T_1, T]$, where $T_1$ is a proper time independent of $R$, which will be specified later in the proof of Proposition~\ref{lem:minorisation}. 

\textit{Step 1}: Notice that the function $J(v,g)$ has zeros in the domain $\mathcal{C}$, and we choose a proper zero point $(v^*,g^*)$ of $J$ and its neighbourhood $\mathcal{N}$. Under some restrictions, $(V_t,G_t)$ can stay in $\mathcal{N}$ by time $T_1$ once entering $\mathcal{N}$. And the restrictions of $(V_t,G_t)$ can only be attached to the reflected OU process $\{G_t\}_{0\leq t\leq T_1}$ since $\{V_t\}$ is totally determined by $\{G_t\}$.

\textit{Step 2}: For the second part, we choose $T_2>0$ independent of $R$ and find a uniform lower bound of transition probability at time $T_2$ with initial data $(v,g)$ for every $(v,g)\in\mathcal{N}$ (see \eqref{e:prob_lower_bound}). Note that $T:=T_1+T_2$ is independent of $R$. Combining two parts yields the desired lower bound \eqref{e:lower_bound}.

\begin{figure}
    \centering
\begin{tikzpicture}[scale=0.9,baseline=2]
    \node at (2.5,4.5) {$\mathcal{C}$};
    \draw [thick] (0,6)--(0,0)--(3,0)--(3,6);
    \draw [thick] (0,5)--(3,5);
    \node at (0,-0.3) {$(0,0)$};
    \node at (3,-0.3) {$(1,0)$};
    \node at (-0.7,5) {$(0,M)$};
    \filldraw [red] (1,4) circle (1pt) ;
    \node at (1,4.3) {\small $(V_0,G_0)$};
    \draw [red][->] (1.1,4)--(1.3,4);
    \draw [blue](0,0) .. controls (1.5,0.7) .. (3,4);
    \filldraw [black] (1.5,1) circle (1pt) ;
    \node at (1.1,1.3) {\small $(v^*,g^*)$};
    \node at (1.5,3) {\footnotesize \textcolor{blue}{$\{J=0\}$}};
    \draw [blue][->] (2.2,2.5)--(2,2.8);
    \node at (1,2) {\footnotesize  \textcolor{blue}{$J>0$}};
    \draw [blue](0.5,1.8) rectangle (1.5,2.2);
    \node at (2.3,0.5) {\footnotesize  \textcolor{blue}{$J<0$}};
    \draw [blue](1.8,0.3) rectangle (2.8,0.7);
\end{tikzpicture}
\begin{tikzpicture}
    [scale=1,baseline=2]
    \draw [->] (0,3)--(1,3);
    \node at (0.5,3.2) {\tiny after};
    \node at (0.5,2.8) {\tiny time $T_1$};
\end{tikzpicture}
\begin{tikzpicture}[scale=0.9,baseline=2]
    \node at (2.5,4.5) {$\mathcal{C}$};
    \draw [thick] (0,6)--(0,0)--(3,0)--(3,6);
    \draw [thick] (0,5)--(3,5);
    \node at (0,-0.3) {$(0,0)$};
    \node at (3,-0.3) {$(1,0)$};
    \node at (-0.7,5) {$(0,M)$};
    \filldraw [red] (1,4) circle (1pt) ;
    \node at (1,4.3) {\small $(V_0,G_0)$};
    \filldraw [black] (1.5,1) circle (1pt) ;
    \node at (1.1,1.7) {\tiny $(v^*,g^*)$};
    \draw [->] (1.5,1) .. controls (1.3,1.1) .. (1.2,1.55);
    \draw [thick] (0.6,0.6) rectangle (2.4,1.4);
    \node at (1.5,0.47) {\tiny $2v_r$};
    \node at (0.4,1) {\tiny $2g_r$};
    \draw (0.6,0.6)--(0.6,0.4);
    \draw [->] (1.25,0.5)--(0.6,0.5);
    \draw (2.4,0.6)--(2.4,0.4);
    \draw [->] (1.75,0.5)--(2.4,0.5);
    \draw (0.25,1.4)--(0.6,1.4);
    \draw [->] (0.4,1.1)--(0.4,1.4);
    \draw (0.25,0.6)--(0.6,0.6);
    \draw [->] (0.4,0.9)--(0.4,0.6);
    \draw (1.1,0.6)--(1.9,1.4);
    \node at (0.8,1.2) {\small $\mathcal{N}$};
    \draw [blue][->] (1,1.2)--(1.6,1.2);
    \draw [blue][->] (0.7,1)--(1.3,1);
    \draw [blue][->] (0.7,0.8)--(1.2,0.8);
    \draw [blue][->] (2.3,1.2)--(1.8,1.2);
    \draw [blue][->] (2.3,1)--(1.6,1);
    \draw [blue][->] (2.3,0.8)--(1.4,0.8);
    \filldraw [red] (1.9,1.1) circle (1pt);
    \draw [red] (1,4) .. controls (1,2.5) and (2.5,2.5) .. (1.9,1.1);
    \node at (2.4,1.8) {\tiny $(V_{T_1},G_{T_1})$};
    \draw [black][->] (1.9,1.1) .. controls (2.2,1.4) .. (2.2,1.65);
\end{tikzpicture}
\begin{tikzpicture}
    [scale=1,baseline=2]
    \draw [->] (0,3)--(1,3);
    \node at (0.5,3.2) {\tiny after};
    \node at (0.5,2.8) {\tiny time $T_2$};
\end{tikzpicture}
\begin{tikzpicture}[scale=0.9,baseline=2]
    \node at (2.5,4.5) {$\mathcal{C}$};
    \draw [thick] (0,6)--(0,0)--(3,0)--(3,6);
    \draw [thick] (0,5)--(3,5);
    \node at (0,-0.3) {$(0,0)$};
    \node at (3,-0.3) {$(1,0)$};
    \node at (-0.7,5) {$(0,M)$};
    \draw [thick](0.6,0.6) rectangle (2.4,1.4);
    \node at (0.8,1.2) {\small $\mathcal{N}$};
    \filldraw [red] (1.9,1.1) circle (1pt);
    \node at (1.75,0.9) {\tiny $\Big(V_{T_1},G_{T_1}\Big)$};
    \draw [thick] (1.2,3) circle (1);
    \node at (1.2,3.6) {supp $\nu$};
    \filldraw [red] (1,2.5) circle (1pt);
    \node at (1,2.7) {\tiny $\Big(V_{T},G_{T}\Big)$};
    \draw [red] (1.9,1.1) .. controls (1,0) and (-1,-1) .. (1,2.5);
\end{tikzpicture}
    \caption{\footnotesize The strategy for the proof of Proposition~\ref{lem:minorisation}. In these coordinate systems, horizontal axes represent $v$-direction, vertical axes represent $g$-direction. In the first illustration, the blue curve represents the zero set of the velocity field $J$, and it splits the domain $\mathcal{C}$ into two parts. The point $(v^*,g^*)$ is a particular zero point of $J$. The red point $(V_0,G_0)$ represents the initial point of the process. In the second illustration, the rectangle domain $\mathcal{N}$ is a neighbourhood of $(v^*,g^*)$. The blue arrows represent the direction of the velocity field $J$. The red curve represents the trajectory of $(V_t,G_t)$ ending at time $T_1$, and the end point $(V_{T_1},G_{T_1})$ belongs to $\mathcal{N}$. In the third illustration, $\nu$ is a probability measure and the circle represents its support. The red curve represents the trajectory of $(V_t,G_t)$ ending at time $T$, and the end point $(V_T,G_T)$ belongs to the support of $\nu$.}
    \label{fig:strategy}
\end{figure}
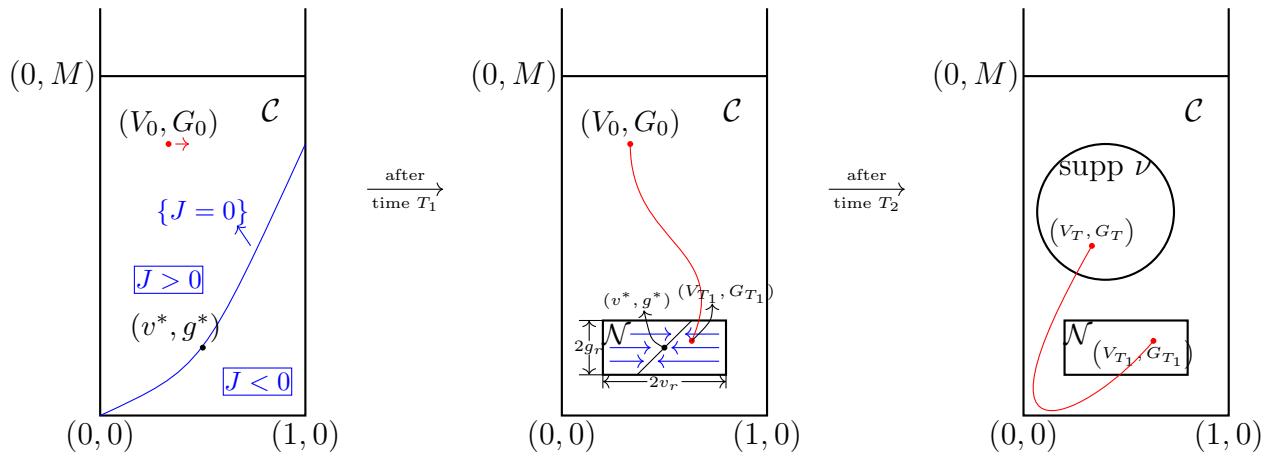

We choose the following zero point of $J$
\begin{equation} \label{e:zero_point}
	(v^*,g^*) = \bigg(\frac{1}{2},\frac{g_L}{2V_E - 1}\bigg),
\end{equation}
and we are going to specify the proper neighbourhood $\mathcal{N}$ later in \eqref{e:N}.

Consider the Markov process $(\widetilde{V}_t)_{t \geq 0}$, which is the same as $V$ but absorbed at $1$. More precisely, it is given by
\begin{equation*}
	\widetilde{V}_{t} := \begin{cases}
		V_{t},\qquad &t<\tau_1;\\
		1,\qquad &t\geq\tau_1,
	\end{cases}
\end{equation*}
where $\tau_1$ is the first jump time defined in \eqref{def:jump-time}. Let $\widetilde{\rho}_t$ be the density of $\widetilde{V}_t$. We have the following lemma.

\begin{lemma} \label{lem:density_continuity_2}
	The map
	\begin{equation*}
		(t,v_0,g_0,v,g) \mapsto \widetilde{\rho}_t^{v_0,g_0}(v,g)
	\end{equation*}
    is smooth on $\mathbb{R^+}\times(0,1)\times\mathbb{R^+}\times(0,1)\times\mathbb{R^+}$. 	Furthermore, there exists a time $T_2>0$ and a point $(v,g)\in(0,1)\times\mathbb{R}^+$ such that \begin{equation} \label{e:positive_density}
	    \widetilde{\rho}_{T_2}^{v^*,\,g^*}(v,g)>0.
	\end{equation}
\end{lemma}

\begin{proof}
	The proof of the smoothness is essentially the same as \cite[Theorem~4.3]{foldes2021method} except that we have the reflective boundary in $g$-direction, which can be dealt in the same way as the proof in Lemma~\ref{lem:density_continuity_1}. We omit the details here.

    Now we choose $T_2= \frac{1}{(2g^*+3)V_E}$, the above result implies that $\widetilde{\rho}_{T_2}^{v_0,g_0}(v,g)$ is continuous with respect to $(v,g)$. Then it suffices to prove
    \begin{equation*}
        \mathbb{P}^{v^*,\,g^*}(T_2 < \tau_1) = \int_{\mathcal{X}}\widetilde{\rho}_{T_2}^{v^*,\,g^*}(v,g)dvdg >0,
    \end{equation*}
    {where $\mathbb{P}^{v^*,g^*}$ is the probability measure conditioned on $\{V_0=v^*,G_0=g^*\}$.} 
    In fact, if $G_t< g^*+1$ for every $0\leq t\leq T_2$, then we have $J(V_t,G_t)< (g^*+1)V_E$ for every $0\leq t\leq T_2$, and hence
	\begin{equation*}
		\widetilde{V}_{T_2} = v^* + \int_0^{T_2} J(V_t,G_t) dt < v^* + T_2(g^*+1)V_E < 1.
	\end{equation*}
	Therefore, we obtain
	\begin{equation*}
		\mathbb{P}^{v^*,g^*}(T_2<\tau_1) \geq \mathbb{P}^{g^*}\Big(\max_{0\leq t\leq T_2}G_t<g^*+1\Big)>\mathbb{P}^{g^*}\Big(\max\limits_{0\leq t\leq T_2}G_t<g^*+1,\min\limits_{0\leq t\leq T_2} G_t >0\Big),
	\end{equation*}
    {where $\mathbb{P}^{g^*}$ is the probability measure conditioned on $\{G_0=g^*\}$.} For the event $\{\max_{0\leq t\leq T_2}G_t<g^*+1,\min_{0\leq t\leq T_2} G_t >0\}$, we can replace the reflected OU process $G_t$ by the OU process without reflective boundary since $G_t$ does not touch the reflective boundary. By the support theorem for diffusion process given in \cite{CPAM1972StroockVaradhan}, we have
    \begin{equation}
    \label{e:apply_support}
        \mathbb{P}^{g^*}\Big(\max_{0\leq t\leq T_2}G_t<g^*+1,\min_{0\leq t\leq T_2} G_t >0\Big)>0.
    \end{equation}
    This completes the proof.
\end{proof}

For the time $T_2$ chosen in the above lemma, we can find a point $(v,g)\in(0,1)\times\mathbb
{R}^+$ such that $\widetilde{\rho}_{T_2}^{v^*,g^*}(v,g)>0$. By the continuity of $\widetilde{\rho}_{T_2}^{v_0,g_0}(v_1,g_1)$ with respect to $(v_0,g_0,v_1,g_1)$, we can find sufficiently small constants $\delta,\varepsilon>0$ and a non-empty open set $\mathcal{K}$ with $(v,g)\in\mathcal{K}\subset \mathcal{X}$ such that
\begin{equation} \label{e:prob_lower_bound}
	\widetilde{\rho}^{v_1,g_1}_{T_2}(v_2,g_2) \geq \varepsilon,\quad \forall\;(v_1,g_1)\in B_{\delta}(v^*,g^*),\quad\forall\;(v_2,g_2)\in\mathcal{K}.
\end{equation}
Note that the choice of the set $\mathcal{K}$ is independent of $R$ since $\mathcal{K}$ depends on the function $\widetilde{\rho}_{T_2}$ only.

Now we choose a small constant $v_r<\delta$ and a sufficiently small constant $g_r$ such that
\begin{equation} \label{e:J*}
	g_r^2 + v_r^2 < \delta^2,\qquad J^*:=\min_{g\in[ g^*-g_r,g^*+g_r]} |J(v^* - v_r,g)| \wedge |J(v^* + v_r,g)| >0,
\end{equation}
where the last condition ensures that $(V_t,G_t)$ moves toward the domain between two lines $v=v^*-v_r$ and $v=v^*+v_r$ with velocity at least $J^*$ in $v$-direction if $G_t\in(g^*-g_r,g^*+g_r)$. Then the neighbourhood $\mathcal{N}$ is determined by (see Figure~\ref{fig:strategy})
\begin{equation}
\label{e:N}
	\mathcal{N} := \left[ v^* - v_r,v^*+v_r \right] \times [ g^*-g_r,g^*+g_r].
\end{equation}

We are going to bound the probability of $\{(V_{T_1},G_{T_1})\in\mathcal{N}\}$, which corresponds to step 1 in Figure~\ref{fig:strategy}. First we provide a technical lemma about $G_t$, which will be used in Lemma~\ref{lem:bound_T_1}.

\begin{lemma}
\label{lem:ou_bound}
    Consider the reflected OU process $G_t$ given by \eqref{sde-g}. For every $M,m>0$, $b>2c>0$, we have
    \begin{equation*}
        \inf_{g_0\in[0,M]}\mathbb{P}^{g_0}(G_t\in[b-2c,b+2c],\;\forall\;t\in[1,1+m]) >0,
    \end{equation*}
    where $\mathbb{P}^{g_0}$ denotes the probability measure conditioned on $\{G_0=g_0\}$.
\end{lemma}
\begin{proof}  
	By Markov property of $G_t$, we have
	\begin{equation*}
     \begin{split}
	    &\phantom{111}\mathbb{P}^{g_0}(G_t\in[ b-2c,b+2c],\;\forall\;1\leq t \leq 1+m)\\
     &\geq \mathbb{P}^{g_0}(G_{1}\in [b-c,b+c]) \cdot \inf_{g\in[b-c,b+c]} \mathbb{P}^g(G_t\in[ b-2c,b+2c],\;\forall\;0\leq t \leq m)\;.
     \end{split}
	\end{equation*}
	Then it suffices to prove the bounds
	\begin{equation} \label{e:Gt_bound}
        \inf_{g\in[b-c,b+c]} \mathbb{P}^g(G_t\in[ b-2c,b+2c],\;\forall\;0\leq t \leq m)>0;
    \end{equation}
    \begin{equation} \label{e:G1_bound}
        \inf_{g_0\in[0,M]} \mathbb{P}^{g_0}(G_{1}\in [b-c,b+c])>0.
	\end{equation}
	For \eqref{e:Gt_bound}, before exiting the interval $[b-2c, b+2c]$, the process $G_t$ is controlled from below and above by two OU processes without boundary reflection starting from $b-c$ and $b+c$ respectively. More precisely, for  
    \begin{equation*}
        G_t^{(1)} = e^{-t}(b-c) + (1 - e^{-t}) + \sqrt{2}\int_0^t e^{-(t-s)} dB_s,\quad G_t^{(2)} = e^{-t}(b+c) + (1 - e^{-t}) + \sqrt{2}\int_0^t e^{-(t-s)} dB_s\;,
    \end{equation*}
    we have
    \begin{equation*}
        \begin{split}
        \inf_{g\in[b-c,b+c]} \mathbb{P}^g(G_t\in[ b-2c,b+2c],\;\forall\;0\leq t \leq m) &\geq \mathbb{P}(b-2c \leq G_t^{(1)} \leq G_t^{(2)}\leq b+2c,\;\forall\;0\leq t \leq m)\\
        &= \mathbb{P}(e^{-t}c-2c \leq G_t^{(3)} \leq 2c - e^{-t}c ,\;\forall\;0\leq t \leq m)\;.
        \end{split}
    \end{equation*}
    where
    \begin{equation*}
        G_t^{(3)}:= (1- e^{-t})(1-b) + \sqrt{2}\int_0^t e^{-(t-s)} dB_s\;.
    \end{equation*}
    Therefore, \eqref{e:Gt_bound} holds by the support theorem for $G_t^{(3)}$. For \eqref{e:G1_bound}, we have
    \begin{equation*}
        \Xi(g_0) := \mathbb{P}^{g_0}(G_{1}\in [b-c,b+c]) = \int_{b-c}^{b+c} \widetilde{\rho}(1,g_0,g) dg.
    \end{equation*}
    Since $(g_0,g)\mapsto \widetilde{\rho}(1,g_0,g) $ belongs to $\mathcal{C}^{\infty}(\mathbb{R}^+ \times \mathbb{R}^+)\cap\mathcal{C}([0,+\infty)\times\mathbb{R}^+)$, we have the continuity of $\Xi$. By the support theorem for reflected diffusion processes established in \cite[Theorem~8]{support_reflection}, we get $\Xi(g_0)>0$ for every fixed $g_0\in[0,M]$. \eqref{e:G1_bound} then follows from the continuity of $\Xi$.
\end{proof}

Now we are ready to prove the main lemma in step 1. 

\begin{lemma}
\label{lem:bound_T_1}
	There exist constants $T_1>0$ and $\gamma(R)>0$ such that we have
	\begin{equation*}
		\inf_{(v_0,g_0)\in\mathcal{C}} \mathbb{P}^{v_0,g_0}((V_{T_1},G_{T_1})\in\mathcal{N}) >\gamma(R)
	\end{equation*}
\end{lemma}
\begin{proof}
	We will choose a proper constant $T_3>0$ and let $T_1:= 1+T_3$ so that we can control $(V_{T_1},G_{T_1})$ in $\mathcal{N}$ if we restrict the reflected OU process in $[g^*-g_r,g^*+g_r]$ for $t\in[1,T_1]$, then it suffices to bound the probability depending on $\{G_t\}_{0\leq t\leq T_1}$ only. 
 
    By the monotonicity of $J$ in $v$-direction, we have
	\begin{equation*}
		\min_{(v,g)\in[0,1)\times[ g^*-g_r,g^*+g_r]\setminus \mathcal{N}} |J(v,g)| = J^*>0.
	\end{equation*}
	Choose $T_3:=\frac{1}{2J^*}$. If $G_t$ is restricted in $[ g^*-g_r,g^*+g_r]$ for $t\in[0,T_3]$, then $(V_t,G_t)$ moves toward $\mathcal{N}$ with velocity at least $J^*$ in $v$-direction and will be trapped in the domain $\mathcal{N}$ by time $T_3$. Therefore, we have the relation
	\begin{equation*}
		   \{ G_t\in[ g^*-g_r,g^*+g_r],\;\forall\;0\leq t \leq T_3 \} \subset \{ (V_{T_3},G_{T_3})\in \mathcal{N}\}.
	\end{equation*}
	Recall that $T_1 = 1+T_3$, we have
	\begin{equation*}
		\inf_{(v_0,g_0)\in\mathcal{C}} \mathbb{P}^{v_0,g_0}((V_{T_1},G_{T_1})\in\mathcal{N}) \geq \inf_{g_0\in[0,M]} \mathbb{P}^{g_0}(G_t\in[ g^*-g_r,g^*+g_r],\;\forall\;1\leq t \leq T_1).
	\end{equation*}
    By Lemma~\ref{lem:ou_bound} with $c=\frac{g_r}{2},b=g^*,m=T_3$, the right hand side of the above inequality is positive and depends only on $R$ since $M(R)=\sqrt{R}+1$ depends on $R$ only.
\end{proof}

Now we have all ingredients to prove Proposition~\ref{lem:minorisation}.

\begin{proof}[Proof of Proposition~\ref{lem:minorisation}]
	
	Choose $T=T_1+T_2 = 1+\frac{1}{(2g^*+3)V_E} + \frac{1}{2J^*}$ which is independent of $R$. Recall the choice of $\mathcal{K}$ satisfying \eqref{e:prob_lower_bound}, the choice of $\mathcal{N}$ given in \eqref{e:N} and the lower bound of the transition probability given in \eqref{e:prob_lower_bound}. For every $A \subset \mathcal{X}$, we have
	\begin{equation*}
			\mathbb{P}^{v_0,g_0} ((V_T,G_T)\in A) \geq  \mathbb{P}^{v_0,g_0}((V_{T_1},G_{T_1})\in \mathcal{N}) \inf_{(v,g)\in\mathcal{N}} \int_{A\cap\mathcal{K}} \widetilde{\rho}^{v,g}_{T_2}(v_2,g_2) dv_2dg_2 \geq  \gamma \varepsilon | A \cap \mathcal{K}|\;,
	\end{equation*}
    where the last inequality follows from \eqref{e:prob_lower_bound} and Lemma~\ref{lem:bound_T_1}.
	
	Finally, let $\nu$ be the probability measure normalised by the Lebesgue measure on $\mathcal{K}$ and choose $\eta= \gamma\varepsilon |\mathcal{K}|$. We have
	\begin{equation*}
		\gamma\varepsilon| A \cap \mathcal{K}| = \eta \nu(A)
	\end{equation*}
    for every $A \subset \mathcal{X}$. Recall that $\mathcal{K}$ and $\varepsilon$ are independent of $R$, and $\gamma$ is independent of the choice of $(v_0,g_0)\in \mathcal{C}$ but depends on $R$. We can conclude that the time $T$ and the measure $\nu$ are independent of $R$, and the factor $\eta(R)$ is independent of $(v_0,g_0)\in\mathcal{C}$. Therefore, the bound \eqref{e:lower_bound} holds.
\end{proof}

\subsection{Exponential ergodicity -- proof of Theorem~\ref{thm:harris}}

Combining the Lyapunov function structure and the minorization condition of $(V_t,G_t)$, we can conclude the proof of Theorem~\ref{thm:harris}.

\begin{proof}[Proof of Theorem~\ref{thm:harris}]
    {Recall the constants $g^*$ defined by \eqref{e:zero_point} and $J^*$ defined by \eqref{e:J*}.} Choose $T= 1+\frac{1}{(2g^*+3)V_E} +\frac{1}{2J^*}$. Lemma~\ref{lem:Lyapunov} showed that there exists a Lyapunov function for the Markov operator $\mathcal{P}_T$. Proposition~\ref{lem:minorisation} showed that the minorization condition also holds for $\mathcal{P}_T$. Thus assumptions~1 and 2 are satisfied with $\mathcal{P} = \mathcal{P}_T$. Therefore, by Theorem~\ref{thm:general_harris}, \eqref{e:harris} holds along the subsequence $t_n = nT$ with $C=1$. More precisely, there exist $\beta>0$ and $\theta \in (0,1)$ such that
    \begin{equation*}
        \|(\pP_T^*)^n \mu_0 - \pi\|_\beta \leq \theta^n \|\mu_0 - \pi\|_\beta
    \end{equation*}
    for every probability distribution $\mu_0$ on $\mathcal{X}$ and every $n \in \mathbb{N}$. 
    
    To extend it to all times, for every $t>0$, we decompose it by $t = nT+ t'$ where $t'\in[0,T)$ and $n\in\mathbb{N}^+$. By Theorem~\ref{thm:harris}, we get
	\begin{equation*}
		\| \mathcal{P}_{t}^* \mu_0  - \pi \|_{\beta} = \| (\mathcal{P}_{T}^*)^n (\mathcal{P}_{t'}^* \mu_0  - \pi) \|_{\beta} \leq \theta^{n} \| \mathcal{P}_{t'}^* \mu_0  - \pi \|_{\beta} .
	\end{equation*}
	Recall the metric defined in \eqref{e:norm}, we obtain
	\begin{equation*}
	    \| \mathcal{P}_{t'}^* \mu_0  - \pi \|_{\beta} \leq \int_{\mathcal{X}} (1+\beta W(v,g)) (\mathcal{P}_{t'}^* |\mu_0-\pi|)(dv,dg) = \int_{\mathcal{X}} (1+\beta\mathcal{P}_{t'}W(v,g))  |\mu_0-\pi|(dv,dg),
	\end{equation*}
	where the first inequality follows from $\mathcal{P}_{t'}^* \pi = \pi$. Note that we can control $\mathcal{P}_{t'}W(v,g)$ as in \eqref{e:lyapunov} by
	\begin{equation*}
	    \mathcal{P}_{t'}W(v,g) \leq (g-1)^2 +1.
	\end{equation*}
	Then we have
	\begin{equation*}
	\begin{split}
	    \int_{\mathcal{X}}& (1+\beta\mathcal{P}_{t'}W(v,g))  |\mu_0-\pi|(dv,dg) \leq \int_{\mathcal{X}} (1+\beta(g-1)^2 +\beta)  |\mu_0-\pi|(dv,dg)\\
	    &\leq \int_{\mathcal{X}} (1+\beta)(1+\beta(g-1)^2) |\mu_0-\pi|(dv,dg) = (1+\beta) \|\mu_0-\pi\|_{\beta}.
	\end{split}
	\end{equation*}
	Combining the above inequalities, we obtain the desired result by choosing $\lambda=-\frac{\log\theta}{T}$ and $C = \theta^{-1}(1+\beta)$.
\end{proof}

\section{Conclusion}\label{sc:5-conclusion}

In this work, we address the exponential convergence to the steady state for the voltage-conductance equation \eqref{eq:FP-PDE}, based on a probabilistic reformulation. In particular, we construct a stochastic process \eqref{sde-v}-\eqref{ic-sde} which is closely related to that in \cite{cai2006kinetic}, the scientific heuristics to derive \eqref{eq:FP-PDE}.

As a by-product, we establish rigorously a link from the constructed stochastic process \eqref{sde-v}-\eqref{ic-sde} to a weak solution to the PDE \eqref{eq:FP-PDE}, which partially justifies the derivation in \cite{cai2006kinetic}. However, it is beyond our focus to develop a complete theory here. The following questions remain open: whether a weak solution in Definition \ref{def:weak-sol} always corresponds to the stochastic process, and whether it has classical regularity at the boundaries. To the best of our knowledge, while various a priori estimates have been obtained and some roadmaps are outlined in \cite{perthame2013voltage}, there has not been a precise definition for a solution to \eqref{eq:FP-PDE}, let alone a proof for the well-posedness. These questions might be subtle for kinetic equations with boundaries; see also \cite{hwang2014fokker,lelievre2020probabilistic} for the classical kinetic Fokker-Planck equation. 

 Beyond the exponential convergence, the nonlinear version of \eqref{eq:FP-PDE} exhibits various phenomena including periodic solutions \cite{caceres2011numerical}, for which our knowledge is limited. Our results give a better understanding for the linear regime, which might serve as a preparation towards rigorously analyzing the nonlinear dynamics. In particular, it might be interesting to see whether the probabilistic reformulation here can be extended to the nonlinear problem and help understand the dynamics.

\section*{Acknowledgement}

The work of Z.Zhou is partially supported by the National Key R\&D Program
of China (Project No. 2021YFA1001200, 2020YFA0712000), and the National Natural Science Foundation of China (Grant No. 12031013, 12171013). W.Xu acknowledges support from National Science Foundation China through grant no.12171008. 
X.Dou and Z.Zhou thank Beno\^ \i t Perthame and Delphine Salort for helpful discussions. 

\begin{appendices}

\section{H\"ormander's theorem}\label{sc:hormander}

In this section, we will briefly introduce H\"ormander's theorem and its applications. H\"ormander's theorem is a powerful tool to verify the hypoellipiticity of differential operators. First we introduce the notion of hypoellipticity.

\begin{definition}
    A differential operator $\mathcal{A}$ is said to be hypoelliptic in a domain $\Omega\subset\mathbb{R}^{n}$ if, $\mathcal{A}u\in \mathcal{C}^{\infty}(U)$ implies $u\in\mathcal{C}^{\infty}(U)$ for every open set $U\subseteq\Omega$.  
\end{definition}
In \cite{hormander1967hypoelliptic}, H\"ormander provided a sufficient condition of the coefficients of the differential operator for hypoellipticity, and this condition is called H\"ormander's condition. To formulate H\"ormander's condition, we recall the Lie bracket between two $\mathcal{C}^{\infty}$ vector fields $V$ and $W$ defined on $\mathbb{R}^n$. The Lie bracket $[V,W]$ is a new vector field given by
\begin{equation*}
    [V,W](x):= DV(x)W(x) - V(x)DW(x),
\end{equation*}
where $DV$ is the derivative matrix of $V$ given by $(DV)_{ij} := \partial_j V_i$. Then H\"ormander's condition can be formulated as follows.

\begin{definition}
    Let $A_0,A_1,\cdots,A_m$ be $\mathcal{C}^{\infty}$ vector fields on $\mathbb{R}^n$. They are said to satisfy H\"ormander's condition in a domain $\Omega\subset\mathbb{R}^n$ if, for every $x\in \Omega$, the vector fields
    \begin{equation*}
        A_i(x)\,(0\leq i \leq m),\; [A_i,A_j](x)\,(0\leq i,j \leq m),\;[\,[A_i,A_j],A_k](x)\,(0\leq i,j,k \leq m),\cdots
    \end{equation*} 
    span $\mathbb{R}^n$.
\end{definition}
Now we are ready to state H\"ormander's theorem given in \cite{hormander1967hypoelliptic}.
\begin{theorem} \label{thm:hormander}
    Consider the differential operator $\mathcal{A}$ defined on $\mathbb{R}^n$ of the form
    \begin{equation}\label{e:operator}
        \mathcal{A}:= \mathcal{A}_0 + \sum_{i=1}^{m} \mathcal{A}_i^2,
    \end{equation}
    where the operators $\mathcal{A}_i(0\leq i\leq m)$ are given by $\mathcal{A}_i:= A_i\cdot\nabla$. Here $A_0,\,A_1,\cdots,A_m$ are $\mathcal{C}^{\infty}$ vector fields on $\mathbb{R}^n$. If $\{A_i\}_{0\leq i\leq m}$ satisfy H\"ormander's condition in a domain $\Omega\subset\mathbb{R}^n$, then the operator $\mathcal{A}$ is hypoelliptic in $\Omega$.
\end{theorem}

For elliptic operator $\mathcal{A}$ of the form \eqref{e:operator}, it is easy to see that $\{A_i\}_{0\leq i\leq m}$ satisfy H\"ormander's condition. Hence, H\"ormander's condition can be viewed as a non-degeneracy condition to generalize the ellipticity for differential operators. As an example, we apply H\"ormander's theorem to our model to verify the hypoellipticity of $\partial_t + \mathcal{L}$ and $\partial_t-\mathcal{L}^*$, where the operators $\mathcal{L}$ and $\mathcal{L}^*$ are given by \eqref{e:L}.

\begin{lemma} \label{lem:hypoelliptic}
	The operators $\partial_t + \mathcal{L}, \partial_t - \mathcal{L}^*$ are hypoelliptic in $\mathbb{R}^+\times(0,V_E)\times\mathbb{R}^+$.
\end{lemma}

\begin{proof}
	We write the operator $\partial_t + \mathcal{L}$ as the form $\mathcal{A}_0+ \mathcal{A}_1^2$, where
	\begin{equation*}
		\mathcal{A}_0:=\partial_{t}+J\partial_{v} + (g-1)\partial_{g},\qquad
        \mathcal{A}_1:= \partial_{g}.
	\end{equation*}
	Let $A_0=(1,J(v,g),g-1)$ and $A_1=(0,0,1)$, we have $\mathcal{A}_i = A_i\cdot\nabla$ for $i=0,1$. Note that the Lie bracket $[A_0,A_1] = (0,V_E-v,1)$. It can be checked that the vector fields $A_0,A_1$ and $[A_0,A_1]$ span $\mathbb{R}^3$ at $(t,v,g)\in\mathbb{R}^+\times(0,V_E)\times\mathbb{R}^+$. Therefore, $A_0,A_1$ satisfy H\"ormander's condition. By Theorem~\ref{thm:hormander}, we obtain the hypoellipticity of $\partial_t + \mathcal{L}$.
	The hypoellipticity of $\partial_t - \mathcal{L}^*$ can be treated in the same way and we omit the details.
\end{proof}

\end{appendices}

\bibliography{bib_torus.bib}

\begin{thebibliography}{10}

\bibitem{baudoin2017bakry}
F.~Baudoin.
\newblock Bakry--{\'e}mery meet villani.
\newblock {\em Journal of Functional Analysis}, 273(7):2275--2291, 2017.

\bibitem{caceres2011analysis}
M.~J. C{\'a}ceres, J.~A. Carrillo, and B.~Perthame.
\newblock Analysis of nonlinear noisy integrate \& fire neuron models: blow-up
  and steady states.
\newblock {\em The Journal of Mathematical Neuroscience}, 1(1):1--33, 2011.

\bibitem{caceres2011numerical}
M.~J. C{\'a}ceres, J.~A. Carrillo, and L.~Tao.
\newblock A numerical solver for a nonlinear fokker--planck equation
  representation of neuronal network dynamics.
\newblock {\em Journal of Computational Physics}, 230(4):1084--1099, 2011.

\bibitem{Cai2004embeded}
D.~Cai, L.~Tao, and D.~W. McLaughlin.
\newblock An embedded network approach for scale-up of fluctuation-driven
  systems with preservation of spike information.
\newblock {\em Proceedings of the National Academy of Sciences},
  101(39):14288--14293, 2004.

\bibitem{cai2006kinetic}
D.~Cai, L.~Tao, A.~V. Rangan, and D.~W. McLaughlin.
\newblock Kinetic theory for neuronal network dynamics.
\newblock {\em Communications in Mathematical Sciences}, 4(1):97--127, 2006.

\bibitem{cai2004effective}
D.~Cai, L.~Tao, M.~Shelley, and D.~W. McLaughlin.
\newblock An effective kinetic representation of fluctuation-driven neuronal
  networks with application to simple and complex cells in visual cortex.
\newblock {\em Proceedings of the National Academy of Sciences},
  101(20):7757--7762, 2004.

\bibitem{Cai2012307}
D.~Cai, L.~Tao, M.~S. Shkarayev, A.~V. Rangan, D.~W. Mclaughlin, and
  G.~Kovačič.
\newblock The role of fluctuations in coarse-grained descriptions of neuronal
  networks.
\newblock {\em Communications in Mathematical Sciences}, 10(1):307 – 354,
  2012.
\newblock Cited by: 4; All Open Access, Bronze Open Access.

\bibitem{cao2021kinetic}
C.~Cao.
\newblock The kinetic fokker--planck equation with general force.
\newblock {\em Journal of Evolution Equations}, 21(2):2293--2337, 2021.

\bibitem{CDZ2022}
J.~A. Carrillo, X.~Dou, and Z.~Zhou.
\newblock A simplified voltage-conductance kinetic model for interacting
  neurons and its asymptotic limit.
\newblock {\em arXiv preprint arXiv:2203.02746}, 2022.

\bibitem{canizo2019hypocoercivity}
J.~A. Cañizo, C.~Cao, J.~Evans, and H.~Yoldaş.
\newblock Hypocoercivity of linear kinetic equations via harris's theorem.
\newblock {\em Kinetic and Related Models}, 13(1):97--128, 2020.

\bibitem{bookMR0192521}
D.~R. Cox and H.~D. Miller.
\newblock {\em The theory of stochastic processes}.
\newblock John Wiley \& Sons, Inc., New York, 1965.

\bibitem{delarue2015global}
F.~Delarue, J.~Inglis, S.~Rubenthaler, and E.~Tanr{\'e}.
\newblock Global solvability of a networked integrate-and-fire model of
  mckean--vlasov type.
\newblock {\em The Annals of Applied Probability}, 25(4):2096--2133, 2015.

\bibitem{delarue2015particle}
F.~Delarue, J.~Inglis, S.~Rubenthaler, and E.~Tanr{\'e}.
\newblock Particle systems with a singular mean-field self-excitation.
  application to neuronal networks.
\newblock {\em Stochastic Processes and their Applications}, 125(6):2451--2492,
  2015.

\bibitem{dolbeault2015hypocoercivity}
J.~Dolbeault, C.~Mouhot, and C.~Schmeiser.
\newblock Hypocoercivity for linear kinetic equations conserving mass.
\newblock {\em Transactions of the American Mathematical Society},
  367(6):3807--3828, 2015.

\bibitem{support_reflection}
H.~Doss and P.~Priouret.
\newblock Support d'un processus de r{\'e}flexion.
\newblock {\em Zeitschrift f{\"u}r Wahrscheinlichkeitstheorie und Verwandte
  Gebiete}, 61(3):327--345, 1982.

\bibitem{dpsz2022dcds}
X.~Dou, B.~Perthame, D.~Salort, and Z.~Zhou.
\newblock Bounds and long term convergence for the voltage-conductance kinetic
  system arising in neuroscience.
\newblock {\em Discrete and Continuous Dynamical Systems}, 43(3\&4):1366--1382,
  2023.

\bibitem{dou2022exponential}
X.~Dou and Z.~Zhou.
\newblock Exponential convergence to equilibrium for a two-speed model with
  variant drift fields via the resolvent estimate, 2022.
\newblock arXiv,2201.12494.

\bibitem{foldes2021method}
J.~Foldes and D.~Herzog.
\newblock The method of stochastic characteristics for linear second-order
  hypoelliptic equations.
\newblock {\em arXiv preprint arXiv:2112.06404}, 2021.

\bibitem{SebastienGadat2013KRM}
S.~Gadat and L.~Miclo.
\newblock Spectral decompositions and $\mathbb{L}^2$-operator norms of toy
  hypocoercive semi-groups.
\newblock {\em Kinetic and Related Models}, 6(2):317--372, 2013.

\bibitem{ha2009applications}
W.~Ha.
\newblock {\em Applications of the reflected Ornstein-Uhlenbeck process}.
\newblock PhD thesis, University of Pittsburgh, 2009.

\bibitem{Hairer_Mattingly_Harris}
M.~Hairer and J.~Mattingly.
\newblock Yet another look at {H}arris’ ergodic theorem for {M}arkov chains.
\newblock {\em Seminar on Stochastic Analysis, Random Fields and Applications
  VI, Progr. Probab.}, 63:109–117, 2011.

\bibitem{harris1956existence}
T.~Harris.
\newblock The existence of stationary measures for certain.
\newblock In {\em Proceedings of the Third Berkeley Symposium on Mathematical
  Statistics and Probability: Held at the Statistical Laboratory, University of
  California, December, 1954, July and August, 1955}, volume~2, page 113. Univ
  of California Press, 1956.

\bibitem{hormander1967hypoelliptic}
L.~H{\"o}rmander.
\newblock Hypoelliptic second order differential equations.
\newblock {\em Acta Mathematica}, 119(1):147--171, 1967.

\bibitem{hwang2014fokker}
H.~J. Hwang, J.~Jang, and J.~J. Vel{\'a}zquez.
\newblock The fokker--planck equation with absorbing boundary conditions.
\newblock {\em Archive for Rational Mechanics and Analysis}, 214(1):183--233,
  2014.

\bibitem{kim2021fast}
J.~Kim, B.~Perthame, and D.~Salort.
\newblock Fast voltage dynamics of voltage--conductance models for neural
  networks.
\newblock {\em Bulletin of the Brazilian Mathematical Society, New Series},
  52(1):101--134, 2021.

\bibitem{kovavcivc2009fokker}
G.~Kova{\v{c}}i{\v{c}}, L.~Tao, A.~V. Rangan, and D.~Cai.
\newblock Fokker-planck description of conductance-based integrate-and-fire
  neuronal networks.
\newblock {\em Physical Review E}, 80(2):021904, 2009.

\bibitem{lelievre2020probabilistic}
T.~Leli{\`e}vre, M.~Ramil, and J.~Reygner.
\newblock A probabilistic study of the kinetic fokker-planck equation in
  cylindrical domains.
\newblock {\em arXiv preprint arXiv:2010.10157}, 2020.

\bibitem{liu2020rigorous}
J.-G. Liu, Z.~Wang, Y.~Zhang, and Z.~Zhou.
\newblock Rigorous justification of the fokker-planck equations of neural
  networks based on an iteration perspective.
\newblock {\em arXiv preprint arXiv:2005.08285}, 2020.

\bibitem{mattingly2002ergodicity}
J.~C. Mattingly, A.~M. Stuart, and D.~J. Higham.
\newblock Ergodicity for sdes and approximations: locally lipschitz vector
  fields and degenerate noise.
\newblock {\em Stochastic processes and their applications}, 101(2):185--232,
  2002.

\bibitem{meyn1994computable}
S.~P. Meyn and R.~L. Tweedie.
\newblock Computable bounds for geometric convergence rates of markov chains.
\newblock {\em The Annals of Applied Probability}, pages 981--1011, 1994.

\bibitem{perthame2013voltage}
B.~Perthame and D.~Salort.
\newblock On a voltage-conductance kinetic system for integrate and fire neural
  networks.
\newblock {\em Kinetic and Related Models}, 6(4):841--864, 2013.

\bibitem{perthame2018derivation}
B.~Perthame and D.~Salort.
\newblock {Derivation of a voltage density equation from a voltage-conductance
  kinetic model for networks of integrate-and-fire neurons.}
\newblock {\em {Communications in Mathematical Sciences}}, 17(5), 2019.

\bibitem{protter}
P.~Protter.
\newblock {\em Stochastic Integration and Differential Equations}.
\newblock Stochastic Modelling and Applied Probability. Springer Berlin
  Heidelberg, 2005.

\bibitem{rangan2006maximum}
A.~V. Rangan and D.~Cai.
\newblock Maximum-entropy closures for kinetic theories of neuronal network
  dynamics.
\newblock {\em Physical review letters}, 96(17):178101, 2006.

\bibitem{rangan2007numerical}
A.~V. Rangan, D.~Cai, and L.~Tao.
\newblock Numerical methods for solving moment equations in kinetic theory of
  neuronal network dynamics.
\newblock {\em Journal of Computational Physics}, 221(2):781--798, 2007.

\bibitem{salort2022convergence}
D.~Salort and D.~Smets.
\newblock {Convergence towards equilibrium for a model with partial diffusion}.
\newblock HAL preprint hal-03845918, Nov. 2022.

\bibitem{CPAM1972StroockVaradhan}
D.~Stroock and S.~R.~S. Varadhan.
\newblock On degenerate elliptic-parabolic operators of second order and their
  associated diffusions.
\newblock {\em Communications on Pure and Applied Mathematics}, 25(6):651--713,
  1972.

\bibitem{villani2009hypocoercivity}
C.~Villani.
\newblock Hypocoercivity. 949-951.
\newblock {\em American Mathematical Soc}, 2009.

\bibitem{yoldacs2022quantitative}
H.~Yolda{\c{s}}.
\newblock On quantitative hypocoercivity estimates based on harris-type
  theorems.
\newblock {\em arXiv preprint arXiv:2203.00096}, 2022.

\end{thebibliography}
\bibliographystyle{abbrv}

\end{document}